\documentclass[12pt, table]{report}

\usepackage{times}
\usepackage[T1]{fontenc}

\usepackage[letterpaper]{geometry}
 \usepackage{tocloft}
 \usepackage[rm, tiny, center, compact]{titlesec}
 \usepackage{indentfirst}
 \usepackage{etoolbox}

\usepackage{tocvsec2}
 \usepackage{tamuconfig}

\usepackage{rotating}

\usepackage{amsmath, amsthm}

\usepackage{graphicx}

\DeclareGraphicsExtensions{.png}

\graphicspath{ {./graphic/} }

\usepackage{afterpage}

\usepackage{longtable}

\usepackage{setspace}

\usepackage{amssymb, amsfonts}
\usepackage{tikz,tikz-cd}
\usepackage{xcolor}

\DeclareMathOperator{\gr}{Gr}
\DeclareMathOperator{\g}{Gal}
\DeclareMathOperator{\fl}{\mathbb{F}\ell}
\DeclareMathOperator{\codim}{codim}
\DeclareMathOperator{\rk}{rank}
\DeclareMathOperator{\init}{in}
\DeclareMathOperator{\gl}{GL}

\theoremstyle{remark}

\numberwithin{equation}{chapter}
\newtheorem{theorem}{Theorem}[chapter]

\newtheorem{lemma}[theorem]{Lemma}
\newtheorem{cor}[theorem]{Corollary}

\newtheorem{alg}[theorem]{Algorithm}
\newtheorem{definition}[theorem]{Definition}
\newtheorem{example}[theorem]{Example}

\newcommand{\I}{\includegraphics{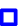}}
\newcommand{\II}{\includegraphics{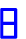}}
\newcommand{\III}{\includegraphics{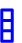}}
\newcommand{\T}{\includegraphics{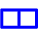}}
\newcommand{\TI}{\includegraphics{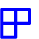}}
\newcommand{\TII}{\includegraphics{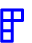}}
\newcommand{\TT}{\includegraphics{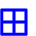}}
\newcommand{\TTI}{\includegraphics{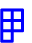}}
\newcommand{\Th}{\includegraphics{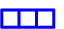}}
\newcommand{\ThI}{\includegraphics{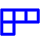}}
\newcommand{\ThII}{\includegraphics{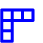}}
\newcommand{\ThT}{\includegraphics{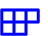}}
\newcommand{\ThTI}{\includegraphics{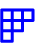}}
\newcommand{\ThTh}{\includegraphics{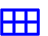}}
\newcommand{\ThThI}{\includegraphics{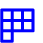}}
\newcommand{\ThThT}{\includegraphics{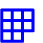}}
\newcommand{\F}{\includegraphics{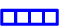}}
\newcommand{\FI}{\includegraphics{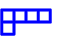}}
\newcommand{\FII}{\includegraphics{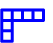}}
\newcommand{\FT}{\includegraphics{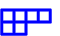}}
\newcommand{\FTI}{\includegraphics{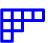}}
\newcommand{\FTT}{\includegraphics{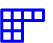}}
\newcommand{\FTh}{\includegraphics{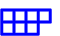}}
\newcommand{\FThI}{\includegraphics{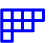}}

\begin{document}

\renewcommand{\tamumanuscripttitle}{Restrictions on Galois groups of Schubert problems}

\renewcommand{\tamupapertype}{Dissertation}

\renewcommand{\tamufullname}{Robert Lee Williams}

\renewcommand{\tamudegree}{Doctor of Philosophy}
\renewcommand{\tamuchairone}{Frank Sottile}

\renewcommand{\tamumemberone}{Laura Matusevich}
\newcommand{\tamumembertwo}{J. Maurice Rojas}
\newcommand{\tamumemberthree}{John Keyser}
\renewcommand{\tamudepthead}{Emil Straube}

\renewcommand{\tamugradmonth}{August}
\renewcommand{\tamugradyear}{2017}
\renewcommand{\tamudepartment}{Mathematics}

\providecommand{\tabularnewline}{\\}

\begin{titlepage}
\begin{center}
\MakeUppercase{\tamumanuscripttitle}
\vspace{4em}

A \tamupapertype

by

\MakeUppercase{\tamufullname}

\vspace{4em}

\begin{singlespace}

Submitted to the Office of Graduate and Professional Studies of \\
Texas A\&M University \\

in partial fulfillment of the requirements for the degree of \\
\end{singlespace}

\MakeUppercase{\tamudegree}
\par\end{center}
\vspace{2em}
\begin{singlespace}
\begin{tabular}{ll}
 & \tabularnewline
& \cr
Chair of Committee, & \tamuchairone\tabularnewline
Committee Members, & \tamumemberone\tabularnewline
 & \tamumembertwo\tabularnewline
 & \tamumemberthree\tabularnewline
Head of Department, & \tamudepthead\tabularnewline

\end{tabular}
\end{singlespace}
\vspace{3em}

\begin{center}
\tamugradmonth \hspace{2pt} \tamugradyear

\vspace{3em}

Major Subject: \tamudepartment \par
\vspace{3em}
Copyright \tamugradyear \hspace{.5em}\tamufullname 
\par\end{center}
\end{titlepage}
\pagebreak{}

\chapter*{ABSTRACT}
\addcontentsline{toc}{chapter}{ABSTRACT} %

\pagestyle{plain} %
\pagenumbering{roman} %
\setcounter{page}{2}

\indent The Galois group of a Schubert problem encodes some structure of its set of solutions. Galois groups are known for a few infinite families and some special problems, but what permutation groups may appear as a Galois group of a Schubert problem is still unknown. We expand the list of Schubert problems with known Galois groups by fully exploring the Schubert problems on $\gr(4,9)$, the smallest Grassmannian for which they are not currently known. We also discover sets of Schubert conditions for any sufficiently large Grassmannian that imply the Galois group of a Schubert problem is much smaller than the full symmetric group.

These results are attained by combining computational exploration with geometric arguments. We use a technique initially described by Vakil to filter out many problems whose Galois group contains the alternating group. We then implement a more computationally intensive algorithm that collects data about the Galois groups of the remaining problems. For each of these, we either gather enough data about elements in the Galois group to determine that it must be the full symmetric group, or we find structure in the set of solutions that restricts the Galois group. Combining the restrictions imposed by the structure of the solutions with the data gathered about the group through the algorithm, we are able to determine the Galois group of these problems as well.

\pagebreak{}

\chapter*{DEDICATION}
\addcontentsline{toc}{chapter}{DEDICATION}  %

\begin{center}
\vspace*{\fill}
Mom, thank you for always believing in me and all that you have sacrificed to give me the opportunity to succeed.

Hien, thank you for always being at my side and supporting me through all of life's trials.
\vspace*{\fill}
\end{center}

\pagebreak{}

\chapter*{ACKNOWLEDGMENTS}
\addcontentsline{toc}{chapter}{ACKNOWLEDGMENTS}  %

\indent I would like to thank Frank Sottile for all of his patience and guidance.

\pagebreak{}
\chapter*{CONTRIBUTORS AND FUNDING SOURCES}
\addcontentsline{toc}{chapter}{CONTRIBUTORS AND FUNDING SOURCES}  %

\subsection*{Contributors}
This work was supported by a dissertation committee consisting of Professors Frank Sottile, advisor, Laura Matusevich and Maurice Rojas of the Department of Mathematics and Professor John Keyser of the Department of Computer Science and Engineering.

The list of all problems to be analyzed was compiled by Professor Frank Sottile. The algorithm described at the end of Chapter 2 was implemented by Christopher Brooks and Professor Frank Sottile. The algorithm described in Chapter 3  was implemented with the assistance of a software library developed by the student and Professors Luis Garcia-Puente, James Ruffo, and Frank Sottile.

All other work conducted for the dissertation was completed by the student independently.

\subsection*{Funding Sources}
Graduate study was supported by a fellowship from Texas A\&M University and grant DMS-1501370 from the National Science Foundation. 
\pagebreak{} %

\phantomsection
\addcontentsline{toc}{chapter}{TABLE OF CONTENTS}  

\begin{singlespace}
\renewcommand\contentsname{\normalfont} {\centerline{TABLE OF CONTENTS}}

\setcounter{tocdepth}{4} %

\setlength{\cftaftertoctitleskip}{1em}
\renewcommand{\cftaftertoctitle}{%
\hfill{\normalfont {Page}\par}}

\tableofcontents

\end{singlespace}

\pagebreak{}

\phantomsection
\addcontentsline{toc}{chapter}{LIST OF FIGURES}  

\renewcommand{\cftloftitlefont}{\center\normalfont\MakeUppercase}

\setlength{\cftbeforeloftitleskip}{-12pt} %
\renewcommand{\cftafterloftitleskip}{12pt}

\renewcommand{\cftafterloftitle}{%
\\[4em]\mbox{}\hspace{2pt}FIGURE\hfill{\normalfont Page}\vskip\baselineskip}

\begingroup

\begin{center}
\begin{singlespace}
\setlength{\cftbeforechapskip}{0.4cm}
\setlength{\cftbeforesecskip}{0.30cm}
\setlength{\cftbeforesubsecskip}{0.30cm}
\setlength{\cftbeforefigskip}{0.4cm}
\setlength{\cftbeforetabskip}{0.4cm} 

\listoffigures

\end{singlespace}
\end{center}

\pagebreak{}

\phantomsection
\addcontentsline{toc}{chapter}{LIST OF TABLES}  

\renewcommand{\cftlottitlefont}{\center\normalfont\MakeUppercase}

\setlength{\cftbeforelottitleskip}{-12pt} %

\renewcommand{\cftafterlottitleskip}{1pt}

\renewcommand{\cftafterlottitle}{%
\\[4em]\mbox{}\hspace{2pt}TABLE\hfill{\normalfont Page}\vskip\baselineskip}

\begin{center}
\begin{singlespace}

\setlength{\cftbeforechapskip}{0.4cm}
\setlength{\cftbeforesecskip}{0.30cm}
\setlength{\cftbeforesubsecskip}{0.30cm}
\setlength{\cftbeforefigskip}{0.4cm}
\setlength{\cftbeforetabskip}{0.4cm}

\listoftables 

\end{singlespace}
\end{center}
\endgroup
\pagebreak{}  

\pagestyle{plain} %
\pagenumbering{arabic} %
\setcounter{page}{1}

\chapter{\uppercase {Introduction}}

Galois groups were first considered in enumerative geometry by Jordan in 1870 \cite{Jordan}. He studied several classical problems and identified structure that prevented their Galois groups from being the full symmetric group. One structure that Jordan studied comes from the Cayley-Salmon Theorem \cite{Cayley, Salmon} which states that a smooth cubic surface in $\mathbb{P}^3$ contains $27$ lines. In \cite{Jordan}, Jordan showed that the incidence structure of the $27$ lines forced the Galois group of this problem to be a subgroup of $E_6$.

This area was revived when Harris studied algebraic Galois groups as geometric monodromy groups \cite{Harris_Gal}, an equivalence that was first discovered by Hermite in 1851 \cite{Hermite}. Harris showed that many problems have the full symmetric group as their Galois group, such as the set of lines that lie on a general hypersurface in $\mathbb{P}^n$ of degree $2n-3$ for $n \geq 4$. In the case of the lines of a cubic surface in $\mathbb{P}^3$, Harris showed that the monodromy group is in fact $E_6$ \cite{Harris_Gal}.

We further explore Galois groups in enumerative geometry by focusing on the Schubert calculus, which is the study of linear subspaces that satisfy prescribed incidence conditions with respect to other general linear subspaces \cite{SchubertCalculus}. The structure of these problems makes them ideal for exploring Galois groups as they are readily modeled on a computer. Typically, the first problem one sees in the Schubert calculus is the problem of four lines:

\begin{example}\label{ex: 4 lines}
	How many lines in $\mathbb{P}^3$ meet four lines in general position? 
	
	We will reference Figure \ref{pic: 4 lines} in solving this problem. A hyperboloid in $\mathbb{P}^3$ is uniquely determined by its ten coefficients up to scale for a total of nine degrees of freedom. Restricting a quadratic to a line gives three linear conditions. Thus, given three lines in general position, such as the green, blue, and red lines in Figure \ref{pic: 4 lines}, there exists a unique hyperboloid such that each line lies on its surface. In fact, the hyperboloid is a doubly-ruled surface, and these three lines lie in one ruling. Therefore, the lines that meet these three lines lie in the other ruling of the hyperboloid. The fourth line will intersect the hyperboloid at two points. The black line in Figure \ref{pic: 4 lines} is an instance of such a line. Thus, the lines that meet all four are the lines in the other ruling of the hyperboloid that meet the fourth line at one of these two intersection points. These are the magenta lines in the Figure \ref{pic: 4 lines}.
\end{example}

\begin{figure}
		\begin{center}
			\includegraphics[scale=0.3]{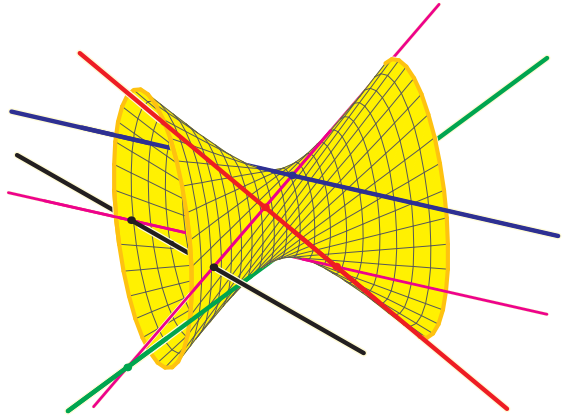}
		\end{center}
		\caption{The problem of four lines.}\label{pic: 4 lines}
\end{figure}

Techniques to study these problems come from many branches of mathematics. On Grassmannians, some structure of the Schubert calculus is reflected in the product structure of the cohomology ring \cite{Fulton_Intersection}. This observation leads to some of the standard notation for the cohomology ring to be adopted for our study. Additionally, combinatorial techniques are useful in both enumerating the number of solutions to a problem and gaining some information about the Galois group \cite{Vakil_checkerboard}. These techniques have been implemented using symbolic computational methods. Computational approaches to these problems have also expanded to using numerical methods to study their Galois groups \cite{Leykin}.

Several advances have been made in classifying the Galois groups that arise in the Schubert calculus of Grassmannians. In \cite{Leykin}, Leykin and Sottile use numerical methods to determine that the Galois group of several simple Schubert problems is the full symmetric group. A simple Schubert problem is one in which all except possibly two conditions impose only a one dimensional restriction. Vakil gave combinatorial criteria for determining when the Galois group contains the alternating group \cite{Vakil_checkerboard}, in which case we say the group is \emph{at least alternating}. He also showed that Schubert problems in the Grassmannian of $2$-planes in $m$-space, $\gr(2,m)$, for $m \leq 16$ and Schubert problems in $\gr(3,m)$ for $m \leq 9$ have an at least alternating Galois group \cite{Vakil_Induction}. Brooks, Mart\'in del Campo, and Sottile expanded this work by showing that any Schubert problem in $\gr(2,m)$ has an at least alternating Galois group \cite{g2n}. Additionally, Sottile and White showed that any Schubert problem in $\gr(3,m)$ and any Schubert problem in $\gr(k,m)$ only involving conditions of the form ``the $k$-plane meets an $\ell$-plane nontrivially with $k+\ell \leq m$'' is doubly transitive \cite{g3n}.

However, it is important to note that Galois groups of Schubert problems need not be at least alternating. The first example of such a problem is given in \cite{Vakil_Induction} and credited to Derksen. In \cite{g3n}, it is found that exactly fourteen Schubert problems in $\gr(4,8)$ have a Galois group that does not contain the alternating group. We will continue this exploration in Chapter 3 by showing that exactly 148 Schubert problems in $\gr(4,9)$ have a Galois group that is not at least alternating. Moreover, we will give some combinatorial conditions which imply that Schubert problems have similar geometric structure restricting their Galois groups and use this structure to organize the 148 problems into eleven different types.

\chapter{BACKGROUND}

We give an introduction to Schubert calculus focusing on Schubert problems on Grassmannians. We also introduce the Galois group of a Schubert problem and explain both Vakil's criterion, which implies a Galois group contains the alternating group \cite{Vakil_Induction}, and the method of sampling from the Galois group via Frobenius elements as outlined in \cite{Campo_Experimentation}. We assume the reader is familiar with the standard ideal-variety correspondence \cite{IVA}.

In Section \ref{alg}, we will cover the algebra that will be used to compute solutions to instances of Schubert problems. In Section \ref{numbtheory}, we will cover the essential number theory and group theory we will use in our sampling algorithm. Section \ref{sp} will introduce the Schubert calculus on the Grassmannian.

\section{Algebra}\label{alg}

We begin by introducing the algebraic ideas behind the algorithm we used following \cite{UAG, IVA}. Choosing local coordinates, we can model instances of Schubert problems with polynomials and calculate their solution sets. Our first step is to develop a method of finding generators for the ideal defined by our Schubert problem that will be useful for large-scale computations.

By taking the indeterminate vector $x:=(x_1,x_2,\dots,x_m)$ and exponent vector $\alpha:=(\alpha_1,\alpha_2,\dots,\alpha_m)$, we may write monomials as $x^\alpha=x_1^{\alpha_1}x_2^{\alpha_2}\cdots x_s^{\alpha_s}$. A \emph{monomial order} on $\mathbb{C}[x]$ is a well-ordering of monomials such that $1$ is minimal and if $x^\alpha \prec x^\beta$, then $x^\alpha x^\gamma \prec x^\beta x^\gamma$. An example of a monomial ordering is the \emph{lexicographic order} in which $x^\alpha \prec x^\beta$ if the last nonzero entry of $\beta-\alpha$ is positive. For instance, 
\[
x_1 \prec x_1^3 \prec x_2 \prec x_1x_2 \prec x_1^{12}x_2 \prec x_2^2.
\]
 For $f \in \mathbb{C}[x]$, we define the \emph{initial term} of $f$, $\init_\prec f$, to be the maximal term of $f$ with respect to $\prec$. For example, if $f=2x_1x_2^4+4x_1^2x_3-3x_2$, then under the lexicographic term order $\init_\prec f = 4x_1^2x_3$. Extending this notion, we define for any ideal $I \subset \mathbb{C}[x]$, the \emph{initial ideal of $I$}, $\init_\prec I := \langle \init_\prec f \mid f \in I \rangle$. 
 
 \begin{definition}
  We call $B=\{b_1, b_2, \dots, b_t\} \subset I$ a \emph{Gr\"obner basis} of $I$ with respect to the monomial ordering $\prec$ if $\langle \init_\prec b_1, \init_\prec b_2, \dots, \init_\prec b_t \rangle = \init_\prec I$. A Gr\"{o}bner basis $B$ is \emph{reduced} if given any $b_i,b_j \in B$ with $i \neq j$, $\init_\prec b_i$ does not divide any term of $b_j$.  
\end{definition}

\begin{lemma}
	If $B$ is a Gr\"obner basis of $I$, then $B$ is a generating set of $I$.
\end{lemma}

\begin{proof}
	Let $I$ be an ideal and $B=\{b_1, \dots, b_t\}$ be a Gr\"obner basis of $I$ with respect to the monomial order $\prec$. By construction, $b_i \in I$ for all $i$, thus we have $\langle b_1, \dots, b_t \rangle \subseteq I$.
	
	To obtain the other inclusion, we assume $I \nsubseteq \langle b_1, \dots, b_t \rangle$ and derive a contradiction. Let $f \in I \setminus \langle b_1, \dots, b_t \rangle$ be such that $\init_\prec f$ is minimal among all $f \in I \setminus \langle b_1, \dots, b_t \rangle$. Since $f \in I$ and $B$ is a Gr\"obner basis of $I$, we must have $\init_\prec f \in \init_\prec \langle b_1, \dots, b_t \rangle$. Therefore, there exists some $g \in \langle b_1, \dots, b_t \rangle$ such that $\init_\prec f = \init_\prec g$. Since $g \in \langle b_1, \dots, b_t \rangle \subset I$, $f \in I$, and $f \notin \langle b_1, \dots, b_t \rangle$, we have $f-g \in I$ and $f-g \notin \langle b_1, \dots, b_t \rangle$. Moreover, by construction $\init_\prec (f-g) \prec \init_\prec f$. However, this contradicts the minimality of $\init_\prec f$. Thus, $B$ must be a generating set of $I$.
\end{proof}

 Gr\"obner bases are very powerful computational tools. In particular, if one has a zero-dimensional variety, then one may obtain a system of polynomials that are useful in determining the points of the variety. When $V(I)$ is zero-dimensional, we say that $I$ is a \emph{zero-dimensional} ideal. 
 Then the following result shows us how we can use Gr\"obner bases to determine the points of $V(I)$.

\begin{theorem}[The Shape Lemma \cite{Shape}]\label{Shape Lemma}
	Let $I \subset \mathbb{C}[x]$ be a zero-dimensional ideal such that $\lvert V(I) \rvert =d$. Suppose there exists a square-free $g_1 \in \mathbb{C}[x_1]$ such that $\deg(g_1)=d$ and $g_1 \in I$. Then the $x_1$-coordinates of the points of $V(I)$ are distinct, there exists  $g_2,\dots,g_m \in \mathbb{C}[x_1]$ with $\deg(g_2),\dots,\deg(g_m) < d$ such that  $\{g_1,x_2-g_2, x_3-g_3,\dots,x_m-g_m \}$ is a Gr\"obner basis for $I$ with respect to a lexicographic order, and $\mathbb{C}[x]/I \cong \mathbb{C}[x_1]/\langle g_1 \rangle$.

\end{theorem}

The polynomial $g_1$ in Theorem \ref{Shape Lemma} is an example of an \emph{eliminant} of $I$; that is, it is a minimal degree univariate polynomial $g(x_1) \in I$. Clearly, $g(x_1)$ vanishes at the $x_1$-coordinates of the points in $V(I)$. On the other hand, the univariate polynomial that vanishes only on the $x_1$-coordinates of the points in $V(I)$ is in $I$ and, by definition of the eliminant, must have degree no smaller than that of an eliminant. Thus, we see the roots of an eliminant are the $x_1$-coordinates of the points in $V(I)$. 

We will now show one way to find an eliminant of a given ideal.  For an ideal $I \subset \mathbb{C}[x]$ satisfying the hypotheses of Theorem \ref{Shape Lemma}, we consider $V=\mathbb{C}[x]/I$ as a $\mathbb{C}$-vector space and define the ``multiplication by $x_i$'' map $m_{x_i}: V \rightarrow V$. Since $V$ is a finite-dimensional vector space, we can choose and order a basis of $V$ to write $m_{x_i}$ as a square matrix. The \emph{characteristic polynomial} of this map may then be computed by $\det(m_{x_i}-tI_r)$ where $m_{x_i}$ is an $r \times r$ matrix and $I_r$ is the $r \times r$ identity matrix. The roots of the characteristic polynomial are the \emph{eigenvalues} of the map.

\begin{theorem}[Corollary 4.6 from Chapter 2 of \cite{UAG}]\label{thm: eliminant}
	Let $I \subset \mathbb{C}[x]$ be zero-dimensional and $V=\mathbb{C}[x]/I$. Then the eigenvalues of the multiplication operator $m_{x_i}$ on $V$ coincide with the $x_i$-coordinates of the points of $V(I)$.
\end{theorem}

We provide a proof of Theorem \ref{thm: eliminant} similar to the proof of a more general result provided in \cite{UAG}.

\begin{proof}
	Without loss of generality, we assume $i=1$. We begin by showing that eigenvalues of $m_{x_1}$ are $x_1$-coordinates of $V(I)$. Let $\lambda$ be an eigenvalue of $m_{x_1}$ and $v \neq 0$ a corresponding eigenvector such that $(x_1-\lambda)v=0$. Assume for contradiction that $\lambda$ is not an $x_1$ coordinate of any point of $\{p_1,\dots,p_n\}=V(I)$. Let $g=x_1-\lambda$. Then $g(p_i) \neq 0$, thus there exists a polynomial $h$ such that $g(p_i)h(p_i)=1$ for all $i$. We now have $1-gh \in I(V(I))$, and thus $(1-gh)^\ell \in I$ for some $\ell \geq 1$. Expanding this expression, we find $1-g\bar{h} \in I$ for some $\bar{h} \in \mathbb{C}[x]$. Thus, $g$ is invertible in $V$. However, $gv=0$ in $V$ with $v \neq 0$, which contradicts the conclusion that $g$ is a unit. Thus, $\lambda$ must be an $x_1$-coordinate of a point in $V(I)$.
	
	We now show that the $x_1$-coordinates of the points in $V(I)$ are eigenvalues of $m_{x_1}$. First, we construct a polynomial, $h$, that divides the characteristic polynomial and vanishes on $V(I)$. Since $V$ is a finite-dimensional, $\{1,x,x^2,\dots\}$ is a linearly dependent set in $V$. Let $h=\sum_{i=0}^q c_iy^i \in \mathbb{C}[y]$ be the minimal degree monic polynomial such that $h(m_{x_1})=0$ in $V$. Since the set of all polynomials in $\mathbb{C}[y]$ that vanish at $m_{x_1}$ is an ideal, all ideals in $\mathbb{C}[y]$ are principal, and $h$ was chosen as a minimal such element, we must have that all polynomials in $\mathbb{C}[y]$ that vanish at $m_{x_1}$ are divisible by $h$. In particular, the characteristic polynomial of $m_{x_1}$ is divisible by $h$. Moreover, by the definition of $V$, $h(m_{x_1})=0$ is equivalent to $h(x_1) \in I$. Therefore $h$ vanishes on $V(I)$. Hence, the characteristic polynomial of $m_{x_1}$ vanishes on $V(I)$ as well.
\end{proof}

We will need one more result on ideals for the theoretical aspects for our algorithm. In particular, we want to be able to find an element that is equivalent to several other elements modulo respective ideals. As long as the ideals involved, say $I_1$ and $I_2$, have elements $r_1 \in I_1$ and $r_2 \in I_2$ such that $r_1+r_2 = 1$, then we can use this equation to find the desired elements.

\begin{theorem}[Chinese Remainder Theorem]\label{thm: Chinese Remainder}
	Let $R$ be a ring and $I_1, \dots, I_s$ ideals of $R$ such that $I_i+I_j=R$ when $i \neq j$. Given elements $r_1,\dots,r_s \in R$, there exists $r \in R$ such that $r \equiv r_i \mod I_i$ for all $i$.
\end{theorem}

\begin{proof}
	Let $R$ and $I_1,\dots,I_s$ be as above. We first prove the theorem for $s=2$. In this case, we desire an $r \in R$ such that $r \equiv r_1 \mod I_1$ and $r \equiv r_2 \mod I_2$. Since $I_1+I_2=R$, there exists $a_1 \in I_1$ and $a_2 \in I_2$ such that $1=a_1+a_2$. Let $r=r_2a_1+r_1a_2$. Since $a_1 \equiv 0 \mod I_1$ and $a_2 \equiv 1 \mod I_1$, we see that $r \equiv r_1 \mod I_1$. Similarly, $r \equiv r_2 \mod I_2$.
	
	We now consider the general case. Since $I_i+I_j=R$ for all $i \neq j$, we may find elements $a_1,\dots,a_s$, where $a_i \in I_i$ such that $a_1+a_i=1$ for all $i>1$. Therefore, $\prod_{i=2}^s (a_1+a_i) =1$, which gives us $I_1+\prod_{i=2}^s I_i = R$. Thus, by the above argument, we may find $\widehat{r}_1 \in R$ such that $\widehat{r}_1 \equiv 1 \mod I_1$ and $\widehat{r}_1 \equiv 0 \mod \prod_{i=2}^s I_i$. Similarly, for $j=2,3,\dots,s$, we may find  $\widehat{r}_j$ such that $\widehat{r}_j \equiv 1 \mod I_j$ and $\widehat{r}_j \equiv 0 \mod \prod_{i\neq j} I_i$. Thus, if we set $r= r_1\widehat{r}_1+\dots+r_s\widehat{r}_s$, we get $r \equiv r_i \mod I_i$ for all $i$.
\end{proof}

\section{Galois Theory}\label{numbtheory}

We will now define what a Galois group is and develop the theory we used to sample from the Galois group of our Schubert problems following \cite{Hungerford, Lang}. Galois groups were classically studied in the context of field extensions. A field $K$ is an \emph{extension} of $F$ if $F$ is a subfield of $K$. The \emph{degree} of the extension $K/F$ is the dimension of $K$ as an $F$-vector space. The structure of $K$ as a field extension of $F$ is encoded by the structure of the field automorphisms of $K$ that are also $F$-module homomorphisms, called \emph{$F$-automorphisms}.

\begin{definition}
	The group of all $F$-automorphisms of $K$ is called the \emph{Galois group} of K over F, denoted $\g(K/F)$. Moreover, the extension $K/F$ is \emph{Galois} if for any $u \in K-F$, there exists some $\sigma \in \g(K/F)$ such that $\sigma(u) \neq u$.
\end{definition}

For any subgroup $G < \g(K/F)$, $\{k \in K \mid \sigma(k)=k \text{ for all } \sigma \in G\}$ is a field. This is called the \emph{fixed field} of $G$ in $K$. The Galois groups we will be looking at will be of field extensions with some structure. An element $u \in K$ is called \emph{algebraic} over $F$ if there exists a polynomial $f \in F[x]$ such that $f(u)=0$. If every element of $K$ is algebraic over $F$, then we say that $K$ is an \emph{algebraic extension} of $F$. For an algebraic element $u \in K$, we say that $u$ is \emph{separable} if the irreducible polynomial in $F[x]$ that vanishes at $u$ may be factored into linear factors in $K[x]$ and all of its roots are simple roots. If every element of $K$ is separable over $F$, then we say that $K$ is a \emph{separable extension} over $F$.

\begin{theorem}[Primitive Element Theorem \cite{Hungerford}]\label{primitive element theorem}
	Let $K/F$ be a finite degree separable extension. Then $K=F(\alpha)$ for some $\alpha \in K$.
\end{theorem}

We will want to take advantage of this property when dealing with the fields $\mathbb{Q}$ and $\mathbb{C}$. As we will now show, all field extensions of these fields are separable.

\begin{lemma}[part (iii) of Theorem III.6.10 in \cite{Hungerford}] \label{f prime 0}
	Suppose $F$ is a field, $f \in F[x]$ is irreducible, and $K$ contains a root of $f$. Then $f$ has no multiple roots in $K$ if and only if $f^\prime \neq 0$.
\end{lemma}

\begin{theorem}[Noted in a remark on page 261 of \cite{Hungerford}]
	Every algebraic extension of a field of characteristic $0$ is separable.
\end{theorem}

\begin{proof}
	Let $F$ be a field of characteristic $0$ and $f \in F[x]$ be an irreducible polynomial of degree $d>0$. Then we have $\deg(f^\prime)=d-1$. Since $F$ is of characteristic $0$, the elements $1,2,\dots,d$ are all distinct in $F$. If $f^\prime$ vanishes at all of these elements, then $(x-a)$ would be a linear factor of $f^\prime$ for $a=1,2,\dots,d$. However, $\deg(f^\prime)=d-1$, so it cannot have $d$ linear factors. Thus, $f^\prime$ does not vanish at one of $1,2,\dots,d$. Since $f^\prime \neq 0$, any extension of $F$ must be separable by Lemma \ref{f prime 0}. 
\end{proof}

We will find it useful to study $\g(K/F)$ by looking at special subgroups of the Galois group. Given a ring $A$ contained in the field $F$, we say that $k \in F$ is \emph{integral} over $A$ if there exists a monic polynomial $f(x) \in A[x]$ such that $f(k)=0$. Furthermore, if we have two rings $A \subset B$, we say that $B$ is \emph{integral} over $A$ if every element of $B$ is integral over $A$. If the ring $A \subset F$ contains every element of $F$ that is integral over $A$, we say that $A$ is \emph{integrally closed} in $F$.

Assume $K/F$ a finite extension, and let $A$ be the smallest integrally closed ring in $F$ containing $\mathbb{Z}$. Let $\mathfrak{a} \subset A$ and $\mathfrak{b} \subset B$ be prime ideals. We say that $\mathfrak{b}$ \emph{lies above} $\mathfrak{a}$ if $\mathfrak{b} \cap A = \mathfrak{a}$. In this case, we have the following commutative diagram where the horizontal maps are the canonical homomorphisms and the vertical maps are inclusions.
\begin{equation}
\begin{tikzcd}
B \arrow{r} & B/\mathfrak{b}  \\
A \arrow{r} \arrow{u} & A/\mathfrak{a} \arrow{u}
\end{tikzcd}
\label{diagram}
\end{equation}
\noindent Moreover, we may show that such a structure always exists. Before we do so, however, we will need to take a brief detour into the structure of rings and modules. 

When $B$ is integral over $A$, then $A[b_1,\dots,b_q]$ for $b_i \in B$ is a finitely-generated $A$-module. To see this, we note that since $b_i$ is integral over $A$, it satisfies a relation of the form $b_i^{t_i}+a_{i,t_i-1}b^{t_i-1}+\dots+a_{i,0}=0$ for some $a_{i,j} \in A$. Thus, any element of $A[b_1,\dots,b_q]$ may be rewritten as $\sum a_\alpha b^\alpha$ where $b^\alpha= (b_1^{\alpha_1},\dots,b_q^{\alpha_q})$ and $\alpha_i < t_i$. We will soon find it useful to apply the following in this setting.

\begin{theorem}[Nakayama's Lemma] \label{thm: Nakayama}
	Let $A$ be a ring, $I$ an ideal contained in all maximal ideals of $A$, and $M$ a finitely generated $A$-module. If $IM=M$, then $M=0$.
\end{theorem}

\begin{proof}
	Let $A, I, \text{ and } M$ be as above, and assume for contradiction that $IM=M\neq 0$. Since $M \neq 0$ is finitely generated, there exists a minimal nonempty set $\{m_1,\dots,m_s\}$ that generates $M$ as an $A$-module. Since $IM=M$, there exists an expression $m_s=a_1m_1+\dots+a_sm_s$ for some $a_i \in I$. Hence, $(1-a_s)m_s=a_1m_1+\dots+a_{s-1}m_{s-1}$.
	
	Since $I$ is contained in every maximal ideal of $A$, $a_s$ is an element of every maximal ideal of $A$. Therefore, $1-a_s$ is in no maximal ideal of $A$ and must be a unit. Let $a=(1-a_s)^{-1}$. Then $m_s=aa_1m_1+\dots+aa_{s-1}m_{s-1}$. Thus, $M$ may be generated by a set of size $s-1$, a contradiction.
\end{proof}

For rings, we will need to consider localization. Let $A$ be a commutative ring and $S$ a subset of $A$ that is closed under multiplication such that $1 \in S$ and $0 \notin S$. Then we define the ring
\[
  S^{-1}A=\Big\{ \frac{a}{s} \Bigm| a \in A, s \in S \Big\}\Big/\sim
\]
where $\frac{a_1}{s_1} \sim \frac{a_2}{s_2}$ if there exists some $s \in S$ such that $s(a_1s_2-a_2s_1)=0$. In this ring, addition and multiplication are defined the same way as the operations are when considering $\mathbb{Q}$ as $(\mathbb{Z}-\{0\})^{-1}\mathbb{Z}$. When $S=A-P$ for some prime ideal $P \subset A$, we write $S^{-1}A=A_P$. We call this the ring \emph{$A$ localized at $P$}. When a ring has a unique maximal ideal, we say it is a \emph{local ring}. 

\begin{theorem}
	If $A$ is a commutative ring and $P$ is a prime ideal of $A$, then $A_P$ is a local ring.
\end{theorem}

\begin{proof}
	Let $m_P$ denote the ideal generated by the image of $P$ under the inclusion map $A \rightarrow A_P$. We first show $1_{A_P}=\frac{1}{1} \notin m_P$. Assume for contradiction that $1_{A_P} \in m_P$. Then there exist $\frac{a_1}{s_1},\dots,\frac{a_q}{s_q} \in A_P$ and $p_1,\dots,p_q \in P$ such that
	\[ \frac{1}{1}=\frac{a_1}{s_1}\cdot\frac{p_1}{1}+\dots+\frac{a_q}{s_q}\cdot\frac{p_q}{1}=\frac{\sum_{i=1}^q \big( a_ip_i \prod_{j\neq i} s_j \big)}{s_1\cdots s_q}
	\] 
	Let $p=\sum_{i=1}^q \big( a_ip_i \prod_{j\neq i} s_j \big) \in P$. Since $\frac{1}{1}=\frac{p}{s_1\cdots s_1}$, there exists some $s \in S$ such that $s(p-s_1\cdots s_q)=0$. Thus, $ss_1s_2\cdots s_q=sp \in P$. However, since $S$ is a multiplicative set, $ss_1s_2\cdots s_q \in S = A-P$, a contradiction.
	
	We have now shown that $m_P$ is a proper ideal of $A_P$. We complete the proof by showing that any other ideal of $A_P$ is either contained in $m_P$ or is all of $A_P$. Let $I$ be an ideal of $A_P$ and $\frac{a}{s} \in I$ such that $\frac{a}{s} \notin m_P$. Since $\frac{a}{s} \notin m_P$, we must have $a \in A-P$. In this case, however, we see that $\frac{s}{a} \in A_P$. Since $\frac{a}{s} \cdot \frac{s}{a} = \frac{as}{sa} = \frac{1}{1}$, we find $I=A_P$.
\end{proof}

We now turn our attention back to Diagram \ref{diagram}. To show such a structure always exists, we will need to explore the relation between the local ring $A_\mathfrak{a}$ and the ring $B_\mathfrak{a}:= (A-\mathfrak{a})^{-1}B$. For a more thorough treatment of the following results, see \cite{Lang}.

\begin{lemma}
	Let $A \subset B$ be rings with $B$ integral over $A$ with no zero divisors and let $\mathfrak{a}$ be an ideal of $A$. Then $B_\mathfrak{a}$ is integral over $A_\mathfrak{a}$.
\end{lemma}

\begin{proof}
	Let $\frac{b}{a} \in B_\mathfrak{a}$ where $b \in B$ and $a \in A-\mathfrak{a}$. Since $B$ is integral over $A$, there exists a relation $b^\ell+a_{\ell-1}b^{\ell-1}+\dots+a_0 =0$ for some $a_i \in A$ and some positive integer $\ell$. Multiplying both sides of the equation by $\frac{1}{a^\ell}$ yields 
	\[
	  \bigg( \frac{b}{a} \bigg)^\ell + \frac{a_{\ell-1}}{a}\bigg(\frac{b}{a}\bigg)^{\ell-1} + \dots + \frac{a_0}{a^\ell} = 0.
	\]
	Thus, $\frac{b}{a}$ is integral over $A_\mathfrak{a}$. Since any element of $B_\mathfrak{a}$ may be written in this form, $B_\mathfrak{a}$ is integral over $A_\mathfrak{a}$.
\end{proof}

\begin{theorem}
	Let $A$ be a ring, $\mathfrak{a}$ a prime ideal, and $B \supset A$ a ring integral over $A$. Then $\mathfrak{a}B \neq B$, and there exists a prime ideal $\mathfrak{b}$ of B lying above $\mathfrak{a}$.
\end{theorem}

\begin{proof}
	We first prove $\mathfrak{a}B \neq B$ by contradiction. Suppose that $\mathfrak{a}B =B$. Then there exists a finite linear combination $1=a_1b_1+\dots+a_\ell b_\ell$ for some $a_i \in \mathfrak{a}$ and $b_i \in B$. We let $B_0:= A[b_1,\dots,b_\ell]$. Then, by construction, we must have $\mathfrak{a}B_0=B_0$. Moreover, since the $b_i$ are integral over $A$, $B_0$ is a finitely generated $A$-module. Thus, by Theorem \ref{thm: Nakayama}, $B_0=0$, a contradiction.
	
	We now turn our attention to the the existence of a prime ideal in $B$ lying above $\mathfrak{a}$. For this, we will make use of the following commutative diagram:
	\begin{equation}
		\begin{tikzcd}
		B \arrow{r} & B_\mathfrak{b}  \\
		A \arrow{r} \arrow{u} & A_\mathfrak{a} \arrow{u}
		\end{tikzcd}
		\label{diagram thm 2.13}
	\end{equation}
	\noindent Let $m_\mathfrak{a}$ denote the maximal ideal of $A_\mathfrak{a}$. Since $m_\mathfrak{a}B_\mathfrak{a}=\mathfrak{a}A_\mathfrak{a}B_\mathfrak{a}=\mathfrak{a}A_\mathfrak{a}B=\mathfrak{a}B_\mathfrak{a}$ and $B_\mathfrak{a}$ is integral over $A_\mathfrak{a}$, by the above argument we have $m_\mathfrak{a}B_\mathfrak{a} \neq B_\mathfrak{a}$. Hence, there exists a maximal ideal of $B_\mathfrak{a}$ containing $m_\mathfrak{a}B_\mathfrak{a}$, say $m_B$. Moreover, we immediately see $m_\mathfrak{a} \subset m_B \cap A_\mathfrak{a}$. Since $m_\mathfrak{a}$ is maximal, this inclusion must be an equality.
	
	Let $\mathfrak{b}=m_B \cap B$. Since $m_B$ is maximal in $B_\mathfrak{a}$, $\mathfrak{b}$ must be prime in $B$. Moreover, $m_\mathfrak{a} \cap A = \mathfrak{a}$. Thus, by following (\ref{diagram thm 2.13}), we find that we must have $\mathfrak{b} \cap A = \mathfrak{a}$.
\end{proof}

 The existence of primes lying above an arbitrary prime in the base ring is useful, as it allows us to examine elements of the Galois group that respect the local structure of these primes. For our extension $K/F$, where $\mathbb{Q} \subset F$, we let $O_K$ and $O_F$ denote the rings of all elements integral over $\mathbb{Q}$ in $K$ and $F$, respectively. Let $P \subset O_K$ be a prime ideal lying above the prime $p \subset O_F$. We define the \emph{decomposition group} of $\g(K/F)$ at $P$ as $D_P:=\{\sigma \in \g(K/F) \mid \sigma(P) = P\}$.
 
 Since the decomposition group is a subgroup of the Galois group, we can gain useful information studying the different decomposition groups. This method turns out to be computationally advantageous, as we can study field extensions of $\mathbb{Q}$ while limiting the size of coefficients that appear in our equations by working over the finite field of $p$ elements, $\mathbb{F}_p$. For this approach, we only need to choose a prime $p$ that does not divide the \emph{discriminant} of $f$, $\Delta_f := \prod_{s_i<s_j} (s_i-s_j)^2$, where the $s_i$ are the roots of $f$.

\begin{theorem}[Dedekind's Theorem \cite{DedekindsTheorem}] \label{thm: Dedekind}
	Consider the monic polynomial $f \in \mathbb{Z}[x]$, irreducible over $\mathbb{Q}$, with splitting field $K=\mathbb{Q}(\alpha_1,\dots,\alpha_n)$ where $f=(x-\alpha_1)\cdots(x-\alpha_n)$. For a prime $p$ that does not divide $\Delta_f$, let $\overline{f}$ be the reduction of $f$ modulo $p$ and $P$ be a prime ideal of $\mathbb{Z}[\alpha_1,\dots,\alpha_n]$ over $\langle p \rangle$. Then there exists a unique element $\sigma_P \in G=\g(K/\mathbb{Q})$, called the \emph{Frobenius element}, such that $\sigma_P(z)\equiv z^p \mod P$ for every $z \in \mathbb{Z}[\alpha_1,\dots,\alpha_n]$. Moreover, if $\overline{f}=g_1\cdots g_s$ with $g_i$ irreducible over $\mathbb{F}_p$, then $\sigma_P$, when viewed as a permutation of the roots of $f$, has a cycle decomposition $\sigma_1\cdots\sigma_s$ with $\sigma_i$ of length $\deg(g_i)$.
\end{theorem}

The following proof is originally due to John Tate. It can be found in Section 4.16 of \cite{DedekindsTheorem}, and is reproduced here for the sake of completeness.

\begin{proof}
	We begin by noting that the roots of $\overline{f}$ are simple, since otherwise there would be two roots, $\alpha_i$ and $\alpha_j$ such that $\alpha_i \equiv \alpha_j \mod p$ which would imply that $p$ divides $\Delta_f$. Thus, the field $K_p=\mathbb{F}_p[\overline{\alpha_1},\dots,\overline{\alpha_n}]$ is a splitting field for $\overline{f}$ where $\overline{z}$ is the residue class of $z$ modulo $P$. The group $G_p=\g(K_p/\mathbb{F}_p)$ is cyclic and generated by the automorphism $\overline{z} \mapsto \overline{z}^p$. Let $D_P$ be the decomposition group of $G$ at $P$. Every automorphism $\sigma \in D_P$ induces an automorphism $\overline{\sigma} \in G_p$ where $\overline{\sigma}(\overline{z}) = \overline{\sigma(z)}$. The homomorphism $\phi : D_P \rightarrow G_p$ given by $\sigma \mapsto \overline{\sigma}$ is injective.
	
	We now show that it is surjective by showing that the fixed field of $\phi (D_P)$ is $\mathbb{F}_p$. Let $a \in \mathbb{Z}[\alpha_1,\dots,\alpha_n]$. By Theorem \ref{thm: Chinese Remainder}, there is an element $z \in \mathbb{Z}[\alpha_1,\dots,\alpha_n]$ such that $z \equiv a \mod P$ and $z \equiv 0 \mod \sigma^{-1}(P)$ for all $\sigma \in G, \sigma \notin D_P$. Therefore, if we define $g:=\prod_{\sigma \in G} (x-\sigma(z)) \in \mathbb{Z}[x]$, we find the reduction of $g$ modulo $p$ is $\overline{g}=x^\ell \prod_{\sigma\in D_P} (x-\overline{\sigma}(\overline{z}))$. Thus, all conjugates of $\overline{z}$ are of the form $\overline{\sigma}(\overline{z})$. Therefore, the fixed field of $\phi(D_P)$ is $\mathbb{F}_p$.

	Let $\sigma_P \in D_P$ be the unique element such that $\overline{\sigma}_P(\overline{z}) \equiv \overline{z}^p$. Then $\sigma_P$ is the unique element of $G$ such that $\sigma_P(z) \equiv x^p$ for every $z \in \mathbb{Z}[\alpha_1,\dots,\alpha_n]$. Since the homomorphism $z \mapsto \overline{z}$ maps the roots of $f$ bijectively onto the roots of $\overline{f}$, we see that $D_P$ and $G_p$ are isomorphic. The cycle decomposition of $\overline{\sigma}$ is determined by the orbits of the action of $G_p$ on the roots of $\overline{f}$ and this group acts transitively on the roots of each polynomial $g_i$, therefore the cycle decomposition of $\sigma_P$ is $\sigma_1 \cdots \sigma_s$ where $\sigma_i$ has length $\deg(g_i)$.
\end{proof}

Our goal is to find the Galois group for families of geometric problems. In practice, we want to specialize the family of problems to specific instances by plugging in values for some parameters, then applying Theorem \ref{thm: Dedekind} to this case. While working over $\mathbb{Q}$, we find that this is a useful method to sample different subgroups of the desired group. This property was first discovered by Hilbert. The interested reader may consult \cite{HIT2} for a modern look at Hilbert's method of proving this and its wider impact on mathematics.

\begin{theorem}[Hilbert's Irreducibility Theorem]\label{thm: HIT}
  If $f(x_1,\dots,x_m,t_1,\dots,t_q)$ is irreducible in  $\mathbb{Z}[x_1,\dots,x_m,t_1,\dots,t_q]$, then there are infinitely many integers $\alpha_1,\dots,\alpha_r$ such that $f(x_1,\dots,x_m,\alpha_1,\dots,\alpha_q)$ is irreducible in $\mathbb{Z}[x_1,\dots,x_m]$.
\end{theorem}

Of course, using these methods is still only sampling from subgroups of the desired group. However, if we are able to determine that certain cycle types are present in the group, then we are able to narrow down the list of possible candidates for our Galois group. In particular, discovering relatively few elements present in the Galois group may be enough to determine if the group is actually the full symmetric group. 

To see this, we will first present a theorem of Jordan where he presents a condition that determines when a subgroup of $S_n$ contains the alternating group, in which case we say the group is \emph{at least alternating}. Since we will be considering groups that can be thought of as permutation groups of $n$ elements, we will use the notation $[n]:=\{1,2,\dots,n\}$. A \emph{partition} of $[n]$ is a collection of disjoint subsets whose union is $[n]$. We call the partitions $\big\{\{1,2,\dots,n\}\big\}$ and $\big\{\{1\},\{2\},\dots,\{n\}\big\}$ \emph{trivial partitions} of $[n]$.

\begin{definition}
	A permutation group $G$ acting on a nonempty set $X$ is called \emph{primitive} if it acts transitively on $X$ and preserves no nontrivial partition of $X$. If $G$ acts transitively on $X$ and does preserve a nontrivial partition, then $G$ is \emph{imprimitive}.
\end{definition}

\begin{theorem}[Jordan \cite{Jordan_Alt}] \label{thm: Jordan Perm}
	If $G$ is a primitive subgroup of $S_n$ and contains a $p$-cycle for some prime number $p<n-2$, then $G$ is at least alternating.
\end{theorem}

It follows almost immediately from this theorem that we can tell if a subgroup of $S_n$ is the entire symmetric group by finding only three cycle types in the group.

\begin{cor} \label{cor: full symmetric}
	If $G$ is a subgroup of $S_n$ that contains an $n$-cycle, an $(n-1)$-cycle, and a $p$-cycle for some prime number $p<n-2$, then $G=S_n$.
\end{cor}

\begin{proof}
	Suppose $G$ satisfies the above hypotheses. We begin by showing that $G$ is primitive. Since $G$ contains an $n$-cycle, $G$ is transitive. Let $X$ be a nontrivial partition of $[n]$. Without loss of generality, assume that the $(n-1)$-cycle in $G$, $\sigma$, cyclically permutes the elements of $[n-1]$. Since $X$ is nontrivial, there exists distinct $a,b,c \in [n]$ such that both $a$ and $b$ are in the same component of $X$, and $c$ and $n$ are in two different components of $X$. Since $G$ is transitive, there exists a permutation $\tau \in G$ such that $\tau(a)=n$. Moreover, $\sigma$ is cyclic on the subset $[n-1]$, thus there exists a positive integer $\ell$ such that $\sigma^\ell\circ\tau(b)=c$. However, since $\sigma^\ell\circ\tau(a)=\sigma^\ell(n)=n$, $G$ must not preserve the partition $X$. We now have that $G$ is a transitive permutation subgroup of $S_n$ that preserves no nontrivial partition of $[n]$, hence $G$ is primitive.
	
	Since $G$ is a primitive subgroup of $S_n$ containing a $p$-cycle for some prime $p < n-2$, $G$ is at least alternating by Theorem \ref{thm: Jordan Perm}. Moreover, $G$ has both an $n$-cycle and an $(n-1)$-cycle. The alternating group only contains permutations of even length and either the $n$-cycle or the $(n-1)$-cycle has odd length, therefore $G=S_n$.
\end{proof}

Since our corollary requires that we find a $p$-cycle for some prime $p<n-2$, it will be useful to be able to tell if a group containing a permutation of some other cycle type will necessarily have a $p$-cycle in it. We note that if we find a permutation whose decomposition into disjoint cycles contains a $p$-cycle for a sufficiently large prime $p$, then we know the group must contain a $p$-cycle as well.

\begin{lemma}\label{lem: cycle elimination}
	Suppose that $\sigma \in S_n$ may be written as a product of disjoint cycles $\sigma=\sigma_1\sigma_2\cdots\sigma_s$ where $\sigma_i$ is an $a_i$-cycle. Furthermore, suppose that $a_s=p$, a prime, and $ \max_{i < s} a_i < p$. Then $\sigma^{(p-1)!}$ is a $p$-cycle.
\end{lemma}

\begin{proof}
	Let $e \in S_n$ denote the identity element. Since disjoint cycles commute, we have
	\begin{align*}
	 \sigma^{(p-1)!} 
	 = \bigg( \prod_{i=1}^{s-1} \sigma_i^{(p-1)!} \bigg) \sigma_s^{(p-1)!}
	 &= \bigg( \prod_{i=1}^{s-1} (\sigma_i^{a_i})^\frac{(p-1)!}{a_i} \bigg) \sigma_s^{(p-1)!} \\
	 &= \bigg( \prod_{i=1}^{s-1} e^\frac{(p-1)!}{a_i} \bigg) \sigma_s^{(p-1)!}
	 = \sigma_s^{(p-1)!}.
	\end{align*}
	Moreover, since $(p-1)!$ and $p$ are relatively prime, $\sigma_s^{(p-1)!}$ is a $p$-cycle.
\end{proof}

\section{Schubert Calculus}\label{sp}

The Schubert calculus is concerned with the cardinality and structure of sets of linear subspaces of a vector space which have specific positions with respect to other fixed linear spaces. In this setting, we can consider the $k$-planes satisfying our conditions as points in the space of $k$-planes in our vector space.

\begin{definition}
	Let $V$ be an $m$-dimensional $\mathbb{C}$-vector space. The \emph{Grassmannian} $\gr(k,V)$ is the set of all $k$-planes in $V$. Alternatively, by choosing a basis, we may write $\gr(k,m):= \gr(k,\mathbb{C}^m)$.
\end{definition}

\begin{definition}
	A \emph{flag}, $F_\bullet$, of a vector space $V$ is a nested sequence of linear subspaces $F_1 \subset F_2 \subset \dots \subset F_m = V$ such that $\dim(F_i)=i$. The space of all flags of $V$ is written as $\fl(V)$. Alternatively, by choosing a basis, we may write $\fl(m):=\fl(\mathbb{C}^m)$.
\end{definition}

Following Chapter 10 of \cite{Fulton_Intersection}, we will express elements of $\gr(k,V)$ in local coordinates via matrices. After choosing a basis of $V \cong \mathbb{C}^m$, we let $M$ be a full rank $k \times m$ matrix and consider the map $M \rightarrow R(M)$ where $R(M)$ is the row span of $M$. In this manner, we see the set of full rank $k \times m$ matrices are mapped to points of $\gr(k,V)$ and two matrices are mapped to the same point if and only if they are equivalent up to action by $\gl(k,\mathbb{C})$. Thus, for every $H \in \gr(k,m)$, there is a unique full rank $k \times m$ matrix in echelon form, $M_H$, such that $M_H \rightarrow H$. Under this map, we see that equations on the space of full rank $k \times m$ matrices give equations on $\gr(k,m)$. In particular, the $k \times m$ echelon matrices whose first $k$ columns form a nonsingular matrix bijectively map to an open subset of $\gr(k,m)$. Thus, $\dim \gr(k,m) = k(m-k)$. 

In a similar manner, we may also represent flags of $V \cong \mathbb{C}^m$ with matrices. Let $M$ be a nonsingular $m \times m$ matrix and $R_\ell(M)$ be the row span of the first $\ell$ rows of $M$. Then
\[
 R_1(M) \subset R_2(M) \subset \dots \subset R_m(M)= \mathbb{C}^m
\]
with $\dim R_i(M) = i$. Thus, we may associate to $F_\bullet$ a full rank matrix $M_F$ such that $R_i(M_F)=F_i$. 

Our primary interest will be in elements of the Grassmannian that satisfy special incidence conditions with respect to a given flag.

\begin{example} \label{ex: Schubert cell}
	We will consider an element of $H \in \gr(4,9)$ that meets the flag $F_\bullet$ in a special way. By choosing the appropriate basis of our vector space, we may assume that $M_F$ is the matrix with $1$s on the anti-diagonal and $0$s elsewhere. Suppose $M_H$ is
	\[
	\begin{pmatrix}
	0 & \textcolor{blue}{0} & 0 & \textcolor{blue}{0} & 0 & 1 & * & * & * \\
	0 & \textcolor{blue}{0} & 0 & \textcolor{blue}{0} & 1 & 0 & * & * & * \\
	0 & \textcolor{blue}{0} & 1 & * & 0 & 0 & * & * & * \\
	1 & * & 0 & * & 0 & 0 & * & * & *
	\end{pmatrix}
	\]
	where $*$ denotes an arbitrary number. The space $F_i$ contains the span of any rows that may be written with only the last $i$ entries nonzero. We see that the last four columns of the matrix contain all of the non-zero entries of its first row. Furthermore, no other row may be expressed as a sum of vectors with only the last four entries nonzero. Thus, $\dim(H \cap F_4)=1$. Continuing in this fashion, we find that $H$ satisfies the following incidence conditions:
	\begin{align}\label{conditons of example 2.22}
	 \dim(H \cap F_4)=1, \dim(H \cap F_5)=2, \dim(H \cap F_7)=3, \dim(H \cap F_9)=4.
	\end{align}
\end{example}

When writing incidence conditions similar to (\ref{conditons of example 2.22}), some conditions we write are satisfied by all subspaces. For example, since $H \in \gr(4,9)$, we must have $\dim(H \cap F_9)=4$ for any flag $F_\bullet$. Similarly, for any $H \in \gr(k,m)$ and any flag $F_\bullet$, we must have $\dim(H \cap F_{m-k+i}) \geq i$. If we let $\lambda_i$ be the largest integer such that $\dim(H \cap F_{m-k+i-\lambda_i})=i$, then $\lambda=(\lambda_1,\dots,\lambda_k)$ completely encodes the position of $H$ with respect to $F_\bullet$. 

Note that $\lambda$ is a \emph{partition}; that is, $\lambda$ is a non-increasing sequence of integers. If, for example, $m=4, k=2, \text{ and } \lambda_1=1$, then we would have the equations $\dim(H \cap F_2)=1$ and $\dim(H \cap F_{4-\lambda_2})=2$. It is clear that both of these conditions can only hold if $\lambda_2 \leq 1 = \lambda_1$. Thus, $\lambda$ satisfies
\begin{equation}
 m-k \geq \lambda_1 \geq \lambda_2 \geq \dots \geq \lambda_k \geq 0. \label{partition restrictions}
\end{equation}
The lower bound $\lambda_k \geq 0$ is required since $\dim(H \cap F_{m-k+k-0}) = k$ for any $F_\bullet \in \fl(m)$ and $H \in \gr(k,m)$. The upper bound $\lambda_1 \leq m-k$ is required since $\dim(H \cap F_{m-k+1-\lambda_1}) = 1$ implies $\dim(F_{m-k+1-\lambda_1}) \geq 1$, hence $m-k+1-\lambda_1 \geq 1$. We may also denote a partition as a \emph{Young diagram}, which is a finite collection of boxes arranged in left justified, non-increasing rows. The partition $\lambda$ is the Young diagram whose $i^{\text{th}}$ row has $\lambda_i$ boxes. For example, the Young diagram for $\lambda=(3,2,0,0)$ is \ThT. %

\begin{definition}\label{def: schubcell}
	For a partition, $\lambda$, satisfying (\ref{partition restrictions}) and a flag $F_\bullet \in \fl(V)$, the \emph{Schubert cell} $\Omega_\lambda^\circ F_\bullet$ in $\gr(k,V)$ is the collection of $k$-planes satisfying
	\begin{align*}
	\Omega_\lambda^\circ F_\bullet := \{ H \in \gr(k,V) \mid &\dim(H \cap F_{m-k+i-\lambda_i}) = i \text{ and } \\
	&\dim(H \cap F_{m-k+i-\lambda_i-1}) < i \text{ for all } i \in [k] \}.
	\end{align*}
	We call $\lambda$ a \emph{Schubert condition} on $\gr(k,V)$ and $F_\bullet$ the \emph{defining flag} of $\Omega_\lambda^\circ F_\bullet$.
\end{definition}

Schubert cells are generators of the cohomology ring of the Grassmannian, and the product structure of this ring can be used to gain information about the intersection of Schubert varieties. For a thorough treatment of this subject, see \cite{Fulton_Intersection}.

We have already seen one example of a Schubert cell: any element $H \in \Omega_{\TTI}^\circ F_\bullet$ in $\g(4,9)$ may be written in the form seen in Example \ref{ex: Schubert cell}. In fact, we can see the partition $\TTI$ in the matrix. If we compare the matrix in Example \ref{ex: Schubert cell}, $M_H$, to a general full-rank $4 \times 9$ matrix in echelon form, we find that $M_H$ has additional zeros which are blue in the example. The shape of these zeros is the same as that of the corresponding partition: $\TTI$.

Alternatively, given the Young diagram and the Grassmannian, we can read the incidence conditions off from the location of the last box of each row. When considering a diagram for a Schubert cell in $\gr(k,m)$, we make a $k \times (m-k)$ block and label each box, as well as a column to the left of the left-most column, starting with a $1$ in the upper-right corner. The labeling is constant along diagonals and increases by one as we move down and to the left. We then place the Young diagram in the upper-left corner of this box. If the last box of the Young diagram's $i^\text{th}$ row is labeled $j$, then for $H \in \Omega_\lambda^\circ F_\bullet$, $\dim(H \cap F_j) = i$ and $\dim(H \cap F_{j-1})<i$. If a row is empty, then we take $j$ to be the number to the left of the $k \times (m-k)$ block in that row. See Figure \ref{fig: one condition} for the example $\Omega_{\TTI}^\circ F_\bullet$ in $\gr(4,9)$ and compare the labels of the last colored box in each row to the incidence conditions given in Example \ref{ex: Schubert cell}.

\begin{figure}
	\begin{center}
		\begin{tabular}{c|c|c|c|c|c|}
			\cline{2-6}
			6 & \cellcolor{blue!25} 5 & \cellcolor{blue!25} 4 & 3 & 2 & 1 \\
			\cline{2-6}
			7 & \cellcolor{blue!25} 6 & \cellcolor{blue!25} 5 & 4 & 3 & 2 \\
			\cline{2-6}
			8 & \cellcolor{blue!25} 7 & 6 & 5 & 4 & 3 \\
			\cline{2-6}
			9 & 8 & 7 & 6 & 5 & 4 \\
			\cline{2-6}
		\end{tabular}
		\caption{\rule{0pt}{0pt}The Young diagram for $(2,2,1,0)$ overlaid on the labeled block for $\gr(4,9)$.}\label{fig: g49 box} \label{fig: one condition}
	\end{center}
\end{figure}

In this example, some of the conditions on $H$ were not explicitly needed to define $\Omega_{\TTI}^\circ F_\bullet$. The condition $\dim(H\cap F_9) = 4$ is implied by $H\in \gr(4,9)$, and the condition $\dim(H \cap F_4) \geq 1$ is implied by $\dim(H \cap F_5) = 2$. In general, the essential conditions on $H$ are those given by the $\lambda_i$ for which $\lambda_i > \lambda_{i+1}$ or $\lambda_k$ if $\lambda_k>0$. In the corresponding Young diagram, these conditions are those given by the boxes in the lower-right corners. 

\begin{definition}
	For a partition, $\lambda$, satisfying (\ref{partition restrictions}) and a flag $F_\bullet \in \fl(V)$, the \emph{Schubert variety} $\Omega_\lambda F_\bullet$ in $\gr(k,V)$ is the collection of $k$-planes of the form
	\[
	\Omega_\lambda F_\bullet := \{ H \in \gr(k,V) \mid \dim(H \cap F_{m-k+i-\lambda_i})  \geq i \}.
	\]
\end{definition}

The Schubert cell $\Omega^\circ_\lambda F_\bullet$ is a dense subset of the Schubert variety $\Omega^\circ_\lambda F_\bullet$ \cite{Kleiman}. Using local coordinates, we have
\begin{equation}
  \dim (H \cap F_i) \geq j \quad \text{if and only if} \quad \rk
  \begin{pmatrix}
  M_H \\ M_{F_i}
  \end{pmatrix}
  \leq k+i-j. \label{Schubert equation}
\end{equation}
For an intersection of Schubert varieties, we may pick a basis such that, for some flag in the intersection $F_\bullet$, $M_F$ is given by the matrix with $1$s on the anti-diagonal and $0$s elsewhere. In this case, the $M_H$ obtained as in Example \ref{ex: Schubert cell} represents all elements in the Schubert cell $\Omega_\lambda^\circ F_\bullet$. Since this is a dense open subset of the Schubert variety, we lose little by narrowing our focus to the affine patch of $\gr(k,m)$ in the form of $M_H$ and using (\ref{Schubert equation}) for the relations $H$ has with respect to the other flags in the intersection. Thus, the entries of $M_H$ marked as $*$ in Example \ref{ex: Schubert cell} must satisfy equations given by a collection of minors vanishing.

\begin{example}\label{ex: equations}
 Consider the flags $F_\bullet,G_\bullet$ of $\mathbb{C}^4$, and choose a basis of $\mathbb{C}^4$ such that $M_F$ is the matrix  with $1$s on the anti-diagonal and $0$s elsewhere. Suppose that $M_G$ is
 \[
 \begin{pmatrix}
 	1 & 1 & 0 & 0 \\
 	0 & 1 & 1 & 0 \\
 	1 & 0 & 0 & 1 \\
 	0 & 0 & 0 & 2 \\
 \end{pmatrix}
 \]
 for this choice of basis. We wish to find the $H \in \gr(2,4)$ such that $H \in \Omega_{\I}F_\bullet \cap \Omega_{\I}G_\bullet$. We know that $M_H$ is of the form
 \begin{equation} \label{H}
  \begin{pmatrix}
   0 & 0 & 1 & x \\
   1 & 0 & 0 & y
  \end{pmatrix}
 \end{equation}
 for some $x,y \in \mathbb{C}$ for an open dense subset of $H \in \Omega_{\I}F_\bullet$. Furthermore, since $H \in \Omega_{\I}G_\bullet$, we have $\dim(H \cap G_2) \geq 1$. Thus, using (\ref{Schubert equation}) yields
 \[
  \rk \begin{pmatrix}
  0 & 0 & 1 & x \\
  1 & 0 & 0 & y \\
  1 & 1 & 0 & 0 \\
  0 & 1 & 1 & 0
  \end{pmatrix}
  \leq 3.
 \]
 This condition is equivalent to determinant of the above matrix vanishing which happens precisely when $x=y$. Therefore, any $H \in \gr(2,4)$ with $M_H$ of the form (\ref{H}) where $x=y$ is in $\Omega_{\I}F_\bullet \cap \Omega_{\I}G_\bullet$.
\end{example}

Since we focus on zero-dimensional varieties, we want to consider intersections of Schubert varieties. The codimension of a Schubert variety in $\gr(k,V)$ is equivalent to the size of the partition $\lvert \lambda \rvert = \sum_{i=1}^k \lambda_i$. This can be seen by counting the number of zeros in the general form of the matrices in the Schubert cell (as in Example \ref{ex: Schubert cell}) and comparing this to the number of zeros that are necessarily in a row-reduced matrix of size $k \times m$.

\begin{definition}
A \emph{Schubert problem}, $\boldsymbol{\lambda}=(\lambda^1,\dots,\lambda^r)$, on $\gr(k,V)$ is a family of intersections of Schubert varieties
\[
\Bigl\{\Omega_{\boldsymbol{\lambda}} F_\bullet := \bigcap_{i=1}^r \Omega_{\lambda^i} F^i_\bullet \mid F^i_\bullet \text{ flags of } V \Bigr\}
\]
 such that $\sum_{i=1}^r \lvert \lambda^i \rvert = k(m-k)$. An instance of a Schubert problem is
 \[
  \Omega_{\boldsymbol{\lambda}} \mathbf{F} 
  := \bigcap_{i=1}^r \Omega_{\lambda^i} F^i_\bullet
 \]
 for a selection of flags $\mathbf{F}=(F^1_\bullet, \dots, F^r_\bullet)$.
\end{definition}

When listing a Schubert problem as a list of Young diagrams, we will use product notation reminiscent of multiplication in the cohomology ring. For example, in $\gr(2,6)$, we may write the Schubert problem $(\I,\I,\I,\T,\TI)$ as $\I^3 \cdot \T \cdot \TI$.

We now study the structure of a Schubert problem. Let $\fl(V)$ denote the space of flags of $V$. We may write $\fl(m)$ when $V=\mathbb{C}^m$. By Kleiman's Transversality Theorem \cite{Kleiman}, there is a dense open subset $U \subset \prod_{i=1}^r \fl(V)$ such that for any $\mathbf{F} \in U$, the intersection $\Omega_{\boldsymbol{\lambda}}\mathbf{F}$ is transverse. Since the codimension of $\Omega_{\boldsymbol{\lambda}}\mathbf{F}$ is $\dim \gr(k,V)$,  it is a zero-dimensional variety. When discussing the number of solutions to a Schubert problem, $\boldsymbol{\lambda}$, we mean the number of points in $\Omega_{\boldsymbol{\lambda}} \mathbf{F}$ for a general choice of flags. However, not every Schubert problem will be of interest to us. When certain pairs of Young diagrams are involved in a Schubert problem, then the problem could possibly be reduced to a problem on a smaller Grassmannian or be an empty intersection.

\begin{lemma} \label{lem: deficient}
	Let $\boldsymbol{\lambda}=(\lambda^1,\dots,\lambda^r)$ be a Schubert problem in $\gr(k,m)$. Then, for any two partitions, say $\lambda$ and $\mu$, we place $\lambda$ in the upper-left corner of the the block for $\gr(k,m)$ and we rotate $\mu$ by $180^\circ$ and place it in the lower-right corner of the block for $\gr(k,m)$. Then we have the following relations:
	\begin{enumerate}
		\item If $\lambda$ and $\mu$  overlap, then $\boldsymbol{\lambda}$ has no solutions. \label{no soln}
		\item If an entire row of the block for $\gr(k,m)$ is filled by $\lambda$ and $\mu$, then $\boldsymbol{\lambda}$ is equivalent to the Schubert problem in $\gr(k-1,m-1)$ obtained by removing that row of $\lambda$ and $\mu$ and not changing the Young diagrams of the other conditions. \label{full row}
		\item If an entire column of the block for $\gr(k,m)$ is filled by $\lambda$ and $\mu$, then $\boldsymbol{\lambda}$ is equivalent to the Schubert problem in $\gr(k,m-1)$ obtained by removing that column of $\lambda$ and $\mu$ and not changing the Young diagrams of the other conditions. \label{full column}
	\end{enumerate}
\end{lemma}

Before proving the lemma, we consider the problem $\ThThI \cdot \FTT \cdot \T^2 \cdot \I$ in $\gr(4,9)$. By choosing $\ThThI$ and $\FTT$, we satisfy relation \ref{full row} of the above lemma as seen in Figure \ref{fig: degenerate conditions}. Thus, by removing those rows of the two partitions, we obtain the equivalent Schubert problem $\ThI \cdot \FT \cdot \T^2 \cdot \I$ in $\gr(3,8)$.

\begin{proof}
	Let $\boldsymbol{\lambda}=(\lambda,\dots,\lambda^r)$ be a Schubert problem in $\gr(k,m)$ with defining flags \linebreak $F^1_\bullet, \dots, F^r_\bullet$, respectively. Let $F_\bullet$ and $G_\bullet$ be the defining flags for $\lambda$ and $\mu$, respectively. We begin by considering relation \ref{no soln}. Suppose $\lambda_r=a, \mu_{k-r+1}=b$, and $a+b>m-k$. Assume for contradiction that there is an $H \in \gr(k,m)$ satisfying $\boldsymbol{\lambda}$, we have $\dim(H \cap F_{m-k+r-a}) \geq r$ and $\dim(H \cap G_{m-r+1-b}) \geq k-r+1$. Since $H$ is $k$-dimensional, we must have $\dim(H \cap  F_{m-k+r-a} \cap G_{m-r+1-b}) \geq 1$. Moreover, $F_\bullet$ and $G_\bullet$ are in general position, so $\dim(F_{m-k+r-a} \cap G_{m-r+1-b}) = \max\{0,(m-k+r-a)+(m-r+1-b)-m\} = \max\{m-k+1-(a+b),0\} = 0$. This is a contradiction. No such $H$ may exist.
	
	We now prove relation \ref{full row}. In this case, we assume we have the same set up as in the previous case except $a+b=m-k$. Then, by going through the same arguments as above, we see $\dim(H \cap F_{m-k+r-a} \cap G_{m-k-r+1-b}) \geq 1$ and $\dim(F_{m-k+r-a} \cap G_{m-k-r+1-b})=1$. Thus, $L:=F_{m-k+r-a} \cap G_{m-k-r+1-b}$ is a line in $H$. Let $V = \mathbb{C}^m/L \cong \mathbb{C}^{m-1}$. We define $\widehat{F}^i_\bullet$ to be the flags where $\widehat{F}^i_j$ is the $j$-dimensional space in the chain $F_1/L \subset \dots \subset F_m/L$. We claim that $H/L$ satisfies the following incidence conditions:
	\begin{enumerate}
		\item $\dim(H \cap F_j)=\dim(H \cap \widehat{F}_j)$ for $j<m-k+r-a$
		\item $\dim(H \cap G_j)=\dim(H \cap \widehat{G}_j)$ for $j<m-k-r+1-b$
		\item $\dim(H \cap F_j)=\dim(H \cap \widehat{F}_j)-1$ for $j \geq m-k+r-a$.
		\item $\dim(H \cap G_j)=\dim(H \cap \widehat{G}_j) -1 $ for $j \geq m-k-r+1-b$
	\end{enumerate} 
	Thus, determining $H \in \gr(k,m)$ satisfying $\boldsymbol{\lambda}$ is equivalent to determining $\widehat{H} \in \gr(k-1,V)$ satisfying the above conditions. Note that, as Young diagrams, these conditions are $\lambda$ and $\mu$ with the $r^\text{th}$ and $(k-r+1)^\text{th}$ row removed, respectively. We also have the conditions $\dim(\widehat{H} \cap \widehat{F}^i_j)=\min\{k-1, \dim(H \cap F^i_j)\}$ for $F^i_\bullet \neq F_\bullet, G_\bullet$. Thus, the Young diagrams for these conditions remain unchanged. This concludes the proof of \ref{full row}.
	
	Finally, we prove relation \ref{full column}. Suppose that this relation is satisfied by $\lambda$ and $\mu$. Then, for some $i$, we have $\dim(H \cap F_{m-i-1})+\dim(H \cap G_i)=\dim(H)$. Thus, we have $H \subset \langle F_{m-i-1}, G_i \rangle := V \equiv \mathbb{C}^{m-1}$. If we take $\widehat{F}^i_\bullet:= F^i_\bullet \cap V$, we get conditions on $H$ in $\gr(k,V)$ with the desired Young diagrams.
\end{proof}

\begin{definition}
	We say that a Schubert problem is \emph{reduced} if it does not satisfy any of the relations in Lemma \ref{lem: deficient}.
\end{definition}

\begin{figure}
	\begin{center}
		\begin{tabular}{c|c|c|c|c|c|}
			\cline{2-6}
			$6$ & \cellcolor{blue!25} $5$ & \cellcolor{blue!25} $4$ & \cellcolor{blue!25} $3$ & $2$ & $1$ \\
			\cline{2-6}
			$7$ & \cellcolor{blue!25} $6$ & \cellcolor{blue!25} $5$ & \cellcolor{blue!25} $4$ & \cellcolor{red!25} $3$ & \cellcolor{red!25} $2$ \\
			\cline{2-6}
			$8$ & \cellcolor{blue!25} $7$ & $6$ & $5$ & \cellcolor{red!25} $4$ & \cellcolor{red!25} $3$ \\
			\cline{2-6}
			$9$ & $8$ & \cellcolor{red!25} $7$ & \cellcolor{red!25} $6$ & \cellcolor{red!25} $5$ & \cellcolor{red!25} $4$ \\
			\cline{2-6}
		\end{tabular}
		\caption{The Young diagrams for $(3,3,1,0)$ (blue) and $(4,2,2,0)$ (red) placed in opposite corners of the labeled block for $\gr(4,9)$.}\label{fig: degenerate conditions}
	\end{center}
\end{figure}

Let $\boldsymbol{\lambda}$ be a Schubert problem on $\gr(k,V)$. Since we are working over the intersection of several Schubert varieties, it is natural to consider these problems as taking place over a product of flag spaces. That is, if we consider 
\begin{align}
 &\pi: Y \rightarrow X  \nonumber \\
 Y = \{(H,F^1_\bullet,F^2_\bullet,&\dots,F^r_\bullet) \mid F^i_\bullet \in \fl(V) \text{ and } H \in \Omega_{\boldsymbol{\lambda}}F_\bullet \} \label{Schubert map} \\
 X = \{(F^1_\bullet,F^2_\bullet,&\dots,F^r_\bullet) \mid F^i_\bullet \in \fl(V) \} \nonumber
\end{align} 

\noindent where $\pi$ is the map that forgets the first coordinate, then the solution set to an instance of a Schubert problem is the set of first coordinates of the preimage of the defining flags. Note that both $X$ and $Y$ are irreducible. Since $X$ is a product of irreducible varieties, it is clear that $X$ is irreducible. To see that $Y$ is irreducible, we note that the fibre over $H \in \gr(k,V)$ is a product of Schubert varieties in the flag varieties. Since each Schubert variety is irreducible, the fibres are as well. The base space $\gr(k,V)$ is also irreducible, thus the total space $Y$ is irreducible.

We now follow \cite{Harris_Gal} in constructing the Galois and monodromy groups of Schubert problems. We let $x \in X$ be a general point. By Kleiman's Transversality Theorem \cite{Kleiman} and the principle of conservation of number \cite[Ch. 10]{Fulton_Intersection}, the fibre over $x$ has $n$ points, $\pi^{-1}(x)=\{y_1, \dots, y_n \}$, where $n$ is the number of solutions to the Schubert problem. If we let $\pi^*:K(X) \rightarrow K(Y)$ be the map on the function fields induced by $\pi$, then by Theorem \ref{primitive element theorem}, there exists $f \in K(Y)$ generating $K(X)$ and satisfying a degree $n$ polynomial
\[
 f^n+g_1f^{n-1}+\dots+g_n=0
\]
where $g_1,\dots,g_n \in K(X)$.

Let $\Delta$ be the field of meromorphic functions in a neighborhood around $x$ modulo the equivalence relation $g \sim h$ if $g=h$ in some  neighborhood of $x$. This is called the germ of meromorphic functions around $x$. Similarly, let $\Delta_i$ be the field of germs of meromorphic functions around $y_i$. Consider the functions $\phi:K(X) \rightarrow \Delta$, the natural inclusion of $K(X)$ into $\Delta$, and $\phi_i:K(Y) \rightarrow \Delta$, the inclusion obtained by composing the natural restriction $K(Y) \rightarrow \Delta_i$ with the map $\Delta_i \rightarrow \Delta$ induced by $\pi$. Then if we let $\tilde{g}_j=\phi(g_j)$ and $\tilde{f}_i=\phi_i(f)$. Then for each $\tilde{f}_i$, we have
\[
 \tilde{f}^d_i+\tilde{g_1}\tilde{f}^{d-1}_i+\dots+\tilde{g}_n = 0.
\]

Note that the $\tilde{f}_i$ are all roots of this polynomial. Let $L \subset \Delta$ be the field generated by the subfields $\phi_i\big(K(Y)\big)$.

\begin{definition}
	The \emph{Galois group} of the Schubert problem $\boldsymbol{\lambda}$ is  $\g\big(L/\phi(K(X))\big)$ where $Y$ and $X$ are defined as in \ref{Schubert map}. This is denoted $\g(\boldsymbol{\lambda})$.
\end{definition}

On the other hand, consider $U \subset X$ a Zariski open subset such that, for all $x \in U$, $\pi^{-1}(x)$ is a set of $n$ distinct points. For any closed path $\gamma:[0,1] \rightarrow U$ with base point $x$ and any point $y_i \in \pi^{-1}(x)$, there is a unique lifting $\tilde{\gamma}_i$ of $\gamma$ to a closed path in $\pi^{-1}(U)$ with $\tilde{\gamma}_i(0)=y_i$. If we define the permutation $\sigma \in S_n$ given by $y_i \mapsto \tilde{\gamma}_i(1)$, then we get a homomorphism $\pi_1(U,x) \rightarrow S_n$. 

\begin{definition}
	The \emph{monodromy group} of a Schubert problem is the image of the homomorphism $\pi_1(U,x) \rightarrow S_n$.
\end{definition}

We observe that the two permutation groups defined above are equivalent.

\begin{theorem}[see page 689 of \cite{Harris_Gal}]\label{monodromy is galois}
	For $\pi:Y \rightarrow X$ as above, the monodromy group equals the Galois group.
\end{theorem}

Moreover, since $Y$ is irreducible, there exists a closed path $\gamma:[0,1] \rightarrow U$ that lifts to a path $\tilde{\gamma}$ in $\pi^{-1}(U)$ with $\tilde{\gamma}(0)=y_1$ and $\tilde{\gamma}(1)=y_i$ for any $i$. Thus, we have the following.

\begin{cor}\label{cor: transitive}
	For any Schubert problem $\boldsymbol{\lambda}$, $\g(\lambda)$ is transitive.
\end{cor}

\begin{example}\label{ex: Galois of 4 lines}
	We consider the Schubert problem $\I^4$ in $\gr(2,4)$. Since $\Omega_{\I}F_\bullet = \{ H \in \gr(2,4) \mid \dim(H \cap F_2) \geq 1 \}$ may be reinterpreted as ``$H$ is a line in $\mathbb{P}^3$ that intersects the line $F_2$ nontrivially'', we see that this is the problem discussed in Example \ref{ex: 4 lines}. As previously, we consider the hyperboloid defined by three of the lines to help us construct the two solution lines.
	
	Since this problem has two solutions and the Galois group must be a transitive permutation group on the set of its solutions by Corollary \ref{cor: transitive}, $\g(\I^4)=S_2$. We can also show this directly by considering the monodromy group of the problem. If we rotate $\ell_4$ by $180^\circ$ about the point $p$ as labeled in Figure \ref{pic: 4 lines with point}, then we see that the solutions, $m_1$ and $m_2$, must move along the surface of the hyperboloid to continue intersecting all four lines. When the rotation of $\ell_4$ is complete, $m_1$ and $m_2$ will have switched places, thus the monodromy group contains the permutation $(m_1, m_2)$. Since this monodromy group must be a subgroup of $S_2$, we have determined that it is all of $S_2$. By Theorem \ref{monodromy is galois}, this is the Galois group of the Schubert problem as well.
\end{example}

\begin{figure}
	\centering
	\includegraphics[scale=0.4]{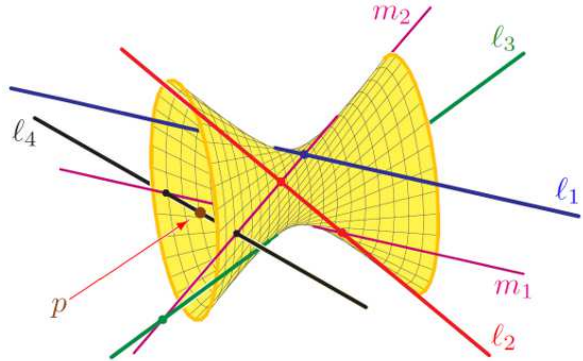} \vspace{-1cm}
	\caption{The problem of four lines with a marked point.}\label{pic: 4 lines with point}
\end{figure}

In \cite{Vakil_checkerboard}, a ``checkerboard tournament'' algorithm is described for determining the number of solutions to a Schubert problem. The algorithm consists of degenerations which transform an intersection of Schubert varieties of the form $\Omega_\lambda F_\bullet \cap \Omega_\mu G_\bullet$ into a union of Schubert varieties $\Omega_\tau F_\bullet$ where $\lvert \tau \rvert = \lvert \lambda \rvert + \lvert \mu \rvert$ using the Geometric Littlewood-Richardson rule \cite{Vakil_checkerboard}. These degenerations are encoded in a combinatorial game that involves moving colored checkers on a checkerboard. The root of a tournament is a Schubert problem, and each degeneration forms the next vertex in a tree. Whenever the game allows for more than one move from a certain variety, the associated vertex has an edge directed away from it corresponding to each of the different resulting varieties. We call these possible continuations of the game \emph{branches}. Eventually, a problem is degenerated to the point that it consists of a single irreducible component corresponding to a single partition. Such vertices are called \emph{leaves}. The number of solutions to a Schubert problem is the number of leaves in this tree. 

Additionally, the structure of the degeneration gives information about the Galois group of the original problem. When a variety in the tree is immediately proceeded by only a single variety, in which case we say it has one \emph{child}, then the Galois group of the parent variety contains the Galois group of the child variety. When a variety has two children, say $X$ and $Y$, the Galois group of the parent contains a subgroup of $\g(X) \times \g(Y)$ such that it maps surjectively onto each component. This allows us to apply the following.

\begin{theorem}[Goursat's Lemma \cite{Goursat}]\label{thm: Goursat}
	Let $G,G^\prime$ be groups and $H$ be a subgroup of $G \times G^\prime$ such that the projections $\pi_1 : H \rightarrow G$ and $\pi_2: H \rightarrow G^\prime$ are surjections with kernels $N$ and $N^\prime$, respectively. Then the image of $H$ in $G/\pi_1(N^\prime) \times G^\prime/\pi_2(N)$ is the graph of an isomorphism $G/\pi_1(N^\prime) \cong G^\prime/\pi_2(N)$.
\end{theorem}

Whenever we come to a branch where the Galois group of both children is at least alternating, we combine Corollary \ref{cor: transitive} with Theorem \ref{thm: Goursat} to see what groups are possible. If both branches have only one leaf, than the parent group is a transitive subgroup of $S_2$, thus it is $S_2$. On the other hand, if one branch has $n$ leaves and the other has $\ell$ leaves with $n \neq \ell$, then the Galois group is a transitive subgroup of $S_{n+\ell}$ containing $A_n \times A_\ell$. In this case, we can show the group is at least alternating. In \cite{Vakil_checkerboard}, Vakil worked through these details and found one other case where the Galois group must be at least alternating.

\begin{theorem}[Theorems 5.2 and 5.10 in \cite{Vakil_Induction}]\label{alg: Vakil}
	Suppose we are given a Schubert problem such that there is a directed tree where each vertex with out-degree two satisfies one of the following:
	\begin{enumerate}
		\item each has a different number of leaves on the two branches
		\item each has one leaf on each branch
		\item there are $n \neq 6$ leaves on each branch and the corresponding Galois group is two-transitive
	\end{enumerate}
	Then the Galois group of the Schubert problem is at least alternating.
\end{theorem}

\chapter{EXPLORATION OF GALOIS GROUPS}

By the results in \cite{g2n} and \cite{g3n}, we know that the Galois group of every Schubert problem in $\gr(2,m)$ is at least alternating and that every Schubert problem in $\gr(3,m)$ is doubly-transitive. For every Schubert problem in $\gr(4,8)$, we either know its exact Galois group or know its Galois group is at least alternating. In the following, we will explore the Schubert problems on $\gr(4,9)$.

In $\gr(4,9)$, there are a total of $31,806$ reduced Schubert problems with at least two solutions. We are limiting our attention to reduced problems on $\gr(4,9)$ since, by Lemma \ref{lem: deficient}, the remaining problems are equivalent to some Schubert problem on $\gr(k,m)$ where $k \leq 4$ and $m \leq 8$. Thus, every problem on $\gr(4,9)$ that we are not treating here is equivalent to some problem examined in either \cite{g2n} or \cite{g3n}. Of the reduced problems on $\gr(4,9)$, we will show that all except $148$ have a Galois group that is at least alternating. We will completely determine the Galois groups of these exceptional cases and group them by the geometry restricting their Galois groups.

\section{Algorithmic sampling of the Galois group}

When the Galois group of a Schubert problem is not at least alternating, we say that it is \emph{deficient}. Our first step in analyzing the Galois groups of all Schubert problems of $\gr(4,9)$ is to sort the problems into two categories: ones with Galois groups that are at least alternating and ones with unknown Galois group. The first sieve we use is applying Vakil's checkerboard tournament algorithm \cite{Vakil_checkerboard} to the list of all Schubert problems and discarding the ones that satisfy Theorem \ref{alg: Vakil}. This lowered the initial set of $31,806$ Schubert problems with up to $1,662,804$ solutions to a set of $225$ Schubert problems with up to $420$ solutions.

For the remainder of the problems, we computed cycle types of random elements in the Galois group of the Schubert problem. We used a \texttt{Python} script \cite{Python} to automate the process of sorting the input and output of several simultaneous \texttt{Singular} programs \cite{Singular}. We used the following algorithm to sample cycle types found in the Galois group of the given Schubert problems.

\begin{alg}[The Frobenius algorithm, see Section 5.4 of \cite{Campo_Experimentation}]\label{alg:Frob}
	~\\
	\textbf{Input}: A Schubert problem $\boldsymbol{\lambda}$ with $n$ solutions, a prime $p$, and an integer $N$ \\
	\textbf{Output}: Either the string ``Galois group is $S_n$'' or a list of cycle types in $\g(\boldsymbol{\lambda})$ \\
	\begin{enumerate}
	  \item Set $c_n=c_{n-1}=c_{\text{prime}}= \text{counter}:=0$ and cycles$:=\{\}$.
	  \item While counter$< N$, do:
	  \begin{enumerate}
		  \item Choose a random point $\mathbf{F}=(F^1_\bullet, \dots, F^r_\bullet) \in \prod_{i=1}^r \fl(m)$.
		  \item Generate the ideal $I$ such that $V(I)= \Omega_{\boldsymbol{\lambda}} \mathbf{F}$ and $I=I\big(V(I)\big)$ using (\ref{Schubert equation}).
		  \item Reduce $I$ modulo $p$ 
		  \item Compute an eliminant $g(x) \in \mathbb{Z}/p\mathbb{Z}[x]$ modulo $p$.
		  \item If $\deg(g)<n$ or $g$ is not square-free, set counter$:=\text{counter}+1$ and return to Step 2(a). 
		  \item Factor $g(x)=g_1g_2\cdots g_t$ in $\mathbb{Z}/p\mathbb{Z}[x]$.
		  \item If $T=1$, set $c_n:=1$.
		  \item If $T=2$ and either $\deg(g_1)=n-1$ or $\deg(g_2)=n-1$, set $c_{n-1}:=1$.
		  \item If $\deg(g_i)=q$ for some $i$ and some prime $q$ such that $\lceil \frac{n}{2} \rceil \leq q < n-2$ and $\deg(g_j) \neq \lceil \frac{n}{2} \rceil$ for any $j\neq i$, set $c_{\text{prime}}:=1$.
		  \item If $c_n=c_{n-1}=c_{\text{prime}}=1$, set counter$:=N$. Otherwise, append \linebreak $\big(\deg(g_1),\dots,\deg(g_T)\big)$ to cycles and set counter$:=\text{counter}+1$ 
	 \end{enumerate}
	 \item If $c_n=c_{n-1}=c_{\text{prime}}=1$, return ``Galois group is $S_n$''. Otherwise, return cycles.
	\end{enumerate}
\end{alg}

Note that this algorithm only works when a Schubert problem has at least six solutions. When the number of solutions is fewer than six, the possible Galois groups are small enough that the algorithm can efficiently determine if it is the full symmetric group via an exhaustive search of cycle types present in the group.

\begin{proof}[Proof of correctness]
	We begin by showing that the calculated number sequences in our algorithm are cycle types of elements in $\g(\boldsymbol{\lambda})$. Since the polynomial we calculate in Step 2(d) is the eliminant of the ideal defining the Schubert problem, it is a monic polynomial whose solutions correspond to $x_1$-coordinates of solutions to the Schubert problem. There are now two possible obstructions to applying Theorem \ref{thm: Dedekind}: the flags we chose may not be in general position leading to an eliminant that is not square-free or a solution that is not in the dense subset of $\gr(4,9)$ that our equations describe (see Example \ref{ex: equations}). By Corollary 1.6 of \cite{Vakil_Induction}, over any finite field there is a positive density of points in the flag space for which we may apply Theorem \ref{thm: Dedekind}, and in practice it is rare to choose flags for which we cannot. If we do pick flags for which the theorem does not apply, they are discarded in Step 2(e). Thus, using Theorem \ref{thm: Dedekind}, we know that $g_{\deg}:=\big(\deg(g_1),\dots,\deg(g_t)\big)$ corresponds to a cycle decomposition of an element of $\g(\boldsymbol{\lambda})$.
	
	We now know that we are calculating cycle types of elements in  $\g(\boldsymbol{\lambda})$ and check that the algorithm's output is as claimed. If one of $c_n, c_{n-1}, c_{\text{prime}}$ is zero by the end of the algorithm, then the algorithm returns all $g_{\deg}$ that were calculated and nothing is left to prove.
	
	Suppose that $c_n=c_{n-1}=c_p=1$ by the conclusion of the algorithm. When the conditions of Steps 2(g) or 2(h) are satisfied, we have found an $n$-cycle or an $(n-1)$-cycle respectively. When the condition of Step 2(i) is satisfied, we have found a permutation, $\sigma$, that is a product disjoint cycles, one of which is a $q$-cycle for some prime $q$ such that $\lceil \frac{n}{2} \rceil \leq q < n-2$ and no other is a $\lceil \frac{n}{2} \rceil$-cycle. Since $\sigma$ is a permutation of $[n]$, the representation of $\sigma$ as a product of disjoint cycles contains only the $q$-cycle and possibly some $a_i$-cycles where $a_i < \lceil \frac{n}{2} \rceil \leq q$. Thus, by Lemma \ref{lem: cycle elimination}, $\sigma^{(q-1)!}$ is a $q$-cycle in $\g(\lambda)$. Hence, by Corollary \ref{cor: full symmetric}, $\g(\lambda) = S_n$.
\end{proof}

We want to highlight that in the above algorithm, we reduce all equations modulo $p$ \emph{before} calculating an eliminant of our ideal. Reducing our equations modulo $p$ before calculating the eliminant rather than after leads to a reduced run time since the coefficients involved in computing an eliminant in this setting grow to the point that simple operations such as ``add two integers'' takes a noticeable amount of time. To highlight this, we ran the above algorithm using the Schubert problem $\F^2 \cdot \III^2 \cdot \I^6$ in $\gr(4,9)$ up until the point where we have factored the eliminant in $\mathbb{Z}/1009\mathbb{Z}[x]$ (ignoring Step $2$(e) in each iteration). When we reduce equations modulo $1009$ prior to computing the eliminant, we were able to calculate and factor $20$ eliminants in $81.79$ seconds. Doing the same calculations on the same problem with the same machine, but not reducing modulo $p$ until after calculating an eliminant, took $2063.61$ seconds. Reducing modulo $p$ prior to calculating the eliminant allows us to calculate cycle types about $25$ times faster in this case. In fact, this speed up seems to be even more significant as the number of solutions to a problem increases. For an example with $20$ solutions, reducing modulo $p$ first is about $444$ times faster. For an example with $30$ solutions, it is about $864$ times faster. When testing an example with $40$ solutions, the reduction gave us an extraordinary boost of speed---the algorithm ran over $96{,}000$ times faster! %

Of the $225$ Schubert problems for which Vakil's algorithm returned an inconclusive result, Algorithm \ref{alg:Frob} found that $77$ problems have the full symmetric group. In fact, we used Algorithm \ref{alg:Frob} on a large set of Schubert problems with relatively few solutions and found that every problem tested that was previously determined to be at least alternating has the full symmetric group.

\begin{theorem}
	All except $148$ reduced Schubert problems on $\gr(4,9)$ with no more than $300$ solutions have the full symmetric group as their Galois group. All Schubert problems on $\gr(4,9)$ with more than $300$ solutions are at least alternating.
\end{theorem}

There are $26{,}051$ reduced Schubert problems on $\gr(4,9)$ with no more than $300$ solutions. Therefore, after we determine the Galois group of the $148$ deficient problems, there are only $5{,}755$ Schubert problems on $\gr(4,9)$ for which the Galois group is not completely determined, and each of these is at least alternating. For each of the $148$ reduced Schubert problems, we will find the Galois group by solving auxiliary Schubert problems on smaller Grassmannians and then building the Galois group of the original problem from the Galois groups of the auxiliary problems. These solutions are built using only a few Schubert problems on $\gr(2,4)$ and $\gr(2,5)$. One of these problems was already solved in Example \ref{ex: Galois of 4 lines}, and $\g(\I^6)$ was first found in \cite{Byrnes}. For the remaining problem, we can find its Galois group using Algorithm \ref{alg:Frob}.

\begin{lemma}\label{lem: aux probs}
	In $\gr(2,4)$, $\g(\I^4)=S_2$. In $\gr(2,5)$, $\g(\I^4 \cdot \T)=S_3$ and $\g(\I^6)=S_5$.
\end{lemma}

\section{Finding the Galois group via auxiliary problems}

We will find that each of the $148$ deficient problems has one of four different Galois groups. Each of these four groups is imprimitive and may be described using a special kind of semidirect product. This construction is studied more carefully in Chapter 7 of \cite{Rotman}.

\begin{definition}
	Let $D$ and $Q$ be groups and $X$ a finite set on which $Q$ acts. Let $K= \prod_{x\in X} D_x$ where $D_x \cong D$ for all $x \in X$. Then the \emph{wreath product} of $D$ by $Q$, denoted $D \wr_X Q$, is $K \rtimes Q$ where $Q$ acts on $K$ by $q \cdot (d_x) = (d_{qx})$ for $q \in Q$ and $(d_x) \in \prod_{x \in X} D_x$.
\end{definition}

If $D$ is a permutation group acting on the set $Y$, $D \wr_X Q$ may be thought of as $Q$ acting on $\lvert X \rvert$ partitions where each partition is a copy of $Y$ being acted upon by a copy of $D$. Thus, if $D$ and $Q$ are both finite, then $\lvert D \wr_X Q \rvert = \lvert D \rvert^{\lvert X \rvert} \lvert Q \rvert$. The groups that we will look at are all wreath products of symmetric groups. With this in mind, we abbreviate the notation by setting $S_n \wr S_m := S_n \wr_{[m]} S_m$. In the following examples, we will write elements of $S_n \wr S_m$ in the form $(\sigma_1,\dots,\sigma_m;\tau)$ where the $\sigma_i \in S_n$ and $\tau \in S_m$. Using this notation, we have $(\sigma_1,\dots,\sigma_m;\tau)(\mu_1,\dots,\mu_m;\upsilon)=(\sigma_1\mu_{\tau(1)},\dots,\sigma_m\mu_{\tau(m)};\tau\upsilon)$.

\begin{example}
	We consider $S_3 \wr S_2$, a permutation group acting on $[6]$. The order of this group is $(3!)^22=72$. The element $(\tau,\mu;\sigma)$ acts on $[6]$ as follows: $\tau$ permutes the elements $\{1,2,3\} \subset [6]$ while $\mu$ permutes $\{4,5,6\} \subset [6]$, then $\sigma$ permutes the parts of the partition $\{\{1,2,3\},\{4,5,6\}\}$. In the following example, we let $\sigma\in S_2$ be the nontrivial permutation, $\tau \in S_3$ be the permutation that exchanges the first and third element, and $\mu \in S_3$ be the permutation that exchanges the second and third element. We now see how $(\tau,\mu;\sigma)$ acts on the ordered sequence $(1,2,3,4,5,6)$.
	\begin{align*}
	 &(\tau,\mu; \sigma) \cdot (1,2,3,4,5,6\big) \\
	 &=\sigma \cdot \big( \tau\cdot(1,2,3), \mu\cdot(4,5,6) \big) \\
	 &=\big(\mu\cdot(4,5,6), \tau\cdot(1,2,3)\big) \\
	 &= (4,6,5,3,2,1)
	\end{align*}
\end{example}

\begin{example}
	We consider $S_2 \wr S_2$. This group has order $2^22=8$. Let $e$ be the identity and $\sigma$ the nontrivial permutation in $S_2$. Then $S_2 \wr S_2$ has an element of order two,
	\[
	(e,e;\sigma)(e,e;\sigma)
	=(e^2, e^2; \sigma^2)
	=(e, e; e)
	\]
	and an element of order four,
	\begin{align*}
	\big((e,\sigma;\sigma)(e,\sigma;\sigma)\big)^2
	& =(e \sigma, \sigma e; \sigma^2)^2 \\
	& =(\sigma, \sigma; e)^2 \\
	& =(\sigma^2, \sigma^2; e^2) 
	=(e,e;e),
	\end{align*}
	with the following relation
	\begin{align*}
	\big((e,e;\sigma)(e,\sigma;\sigma)\big)^2
	& = (e\sigma, e^2; \sigma^2)^2 \\
	& = (\sigma, e; e)^2 \\
	& = (\sigma^2, e^2; e^2)
	=(e,e;e).
	\end{align*}
	Putting these facts together, we may write $S_2 \wr S_2$ in the form $\langle s,r \mid r^4=s^2=(sr)^2=e_\wr \rangle$ where $e_\wr$ is the identity element $(e,e;e)$. Thus, $S_2 \wr S_2 \cong D_4$, the group of symmetries of the square.
	
	We can further see this relation in Figure \ref{fig: S2 wr S2 is D4}. The action of $D_4$ on the vertices of a square may permute the red vertices amongst themselves, permute the blue vertices amongst themselves, or make the set of red vertices switch places with the set of blue vertices. In this manner, we see $D_4$ acting on the four vertices and their red-blue partition in the same way as $S_2 \wr S_2$.
\end{example}

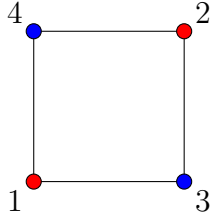
\begin{figure}
	\begin{center}
		\begin{tikzpicture}
		\draw (0,0) rectangle (2,2);
		\draw[fill=red]  (0,0) circle (0.1)
						 (2,2) circle (0.1);
		\draw[fill=blue] (2,0) circle (0.1)
						 (0,2) circle (0.1);
		\node at (-.25,-.25) {$1$};
		\node at (2.25, 2.25) {$2$};
		\node at (-.25, 2.25) {$4$};
		\node at (2.25,-.25) {$3$};
		\end{tikzpicture}
	\end{center}
	\caption{The action of $D_4$ on the vertices of the square preserves the red-blue partition.}\label{fig: S2 wr S2 is D4}
\end{figure}

Each of the remaining $148$ Schubert problems has a geometric structure that allows us to build an auxiliary subproblem. This auxiliary subproblem is itself a Schubert problem on a smaller Grassmannian, and every solution to the original Schubert problem contains a subspace that is a solution to the auxiliary problem. Moreover, for any solution to this subproblem, we are able to find another Schubert problem in the original space modulo the space of the subproblem. These two problems fit together in such a way that every solution to the original Schubert problem satisfies these two subproblems. Furthermore, for every pair of solutions to the two subproblems, we may find a solution to the original problem containing them. This geometric structure will force the Galois group of the whole problem to be a subgroup of the wreath product of the Galois groups of the subproblems.

To properly define these subproblems, we will need a special operation on partitions. Given partitions $\lambda,\mu$, we define $\overline{\lambda-\mu}$ to be coordinate-wise subtraction of $\mu$ from $\lambda$ followed by rearranging the numbers to keep the sequence non-increasing if necessary. In the case $\overline{\lambda-\lambda}$, we get the trivial condition. When we speak of a condition $\lambda$ on $\gr(k_1,m_1)$ as a condition on $\gr(k_2,m_2)$, we mean the condition in $\gr(k_2,m_2)$ whose Young diagram is the same as the Young diagram of $\lambda$ in $\gr(k_1,m_1)$.

\begin{example}
	Let $\lambda=(3,2,1,0)$ and $\mu=(3,0,0,0)$ be conditions on $\gr(4,9)$. As Young diagrams, we have $\lambda= \ThTI$ and $\mu=\III$. Then we have 
	\[
	\overline{\lambda-\mu}=\overline{(3,2,1,0)-(3,0,0,0)}= \overline{(0,2,1,0)}=(2,1,0,0).
	\]
	The equivalent statement with Young diagrams is $\overline{\ThTI-\III}=\TI$.
	
	Viewing $\lambda$ as a condition on $\gr(3,8)$ would preserve the Young diagram $\lambda=\ThTI$, but would change the coordinate representation $\lambda=(3,2,1)$. Similarly, viewing $\overline{\lambda-\mu}$ as a condition on $\gr(2,5)$ would yield $\overline{\lambda-\mu}=(2,1)=\TI$.
\end{example}

In the following, we will speak of the codimension of some space in several different ambient spaces. In order to avoid confusion, we will use the notation $\codim_V H$ to denote the codimension of $H$ with respect to $V$. The first of the eight relations we will show is when the Schubert problem contains four conditions: two containing each of the conditions in Figure \ref{fig: relation 1}.

\begin{figure}[h]
\centering
\begin{minipage}{0.4\textwidth}
	\begin{tabular}{|c|c|c|c|}
		\hline
		\cellcolor{blue!25} m-k & m-k-1& \dots & 1 \\
		\hline
		\cellcolor{blue!25} m-k+1 & m-k & \dots & 2 \\
		\hline
		\cellcolor{blue!25} \vdots & \vdots & $\ddots$ & \vdots \\
		\hline
		\cellcolor{blue!25} m-2 & m-3 & \dots & k-1 \\
		\hline
		m-1 & m-2 & \dots  & k \\
		\hline
	\end{tabular}
\end{minipage}
\begin{minipage}{0.05\textwidth}
	~
\end{minipage}
\begin{minipage}{0.4\textwidth}
	\begin{tabular}{|c|c|c|c|}
		\hline
		\cellcolor{blue!25} m-k & \cellcolor{blue!25} \dots & \cellcolor{blue!25} 2 & 1 \\
		\hline
		m-k+1 & \dots & 3 & 2 \\
		\hline
		\vdots & $\ddots$ & \vdots & \vdots \\
		\hline
		m-1 & \dots & k+1 & k \\
		\hline
	\end{tabular}
\end{minipage}
\caption{The conditions required by the relation in Theorem \ref{thm: relation 1}.}\label{fig: relation 1}
\end{figure}

\begin{theorem}\label{thm: relation 1}
	Let $\boldsymbol{\lambda}=(\lambda^1,\dots,\lambda^r)$ be a reduced Schubert problem on $\gr(k,m)$ with $3 \leq k \leq m-3$. Let $\mu^1=(1,1,\dots,1,0)$ and $\mu^2=(m-k-1,0,0,\dots,0)$ be conditions on $\gr(k,m)$ and suppose $\mu^1 \subset \lambda^1,\lambda^2$ and $\mu^2 \subset \lambda^3,\lambda^4$. Then, for the Schubert problem
	\[ 
	\boldsymbol{\widehat{\lambda}}=(\overline{\lambda^1-\mu^1},\overline{\lambda^2-\mu^1},\overline{\lambda^3-\mu^2},\overline{\lambda^4-\mu^2},\lambda^5,\lambda^6,\dots,\lambda^r) 
	\]
	in $\gr(k-2,m-4)$, we have
	$\g(\boldsymbol{\lambda}) \leq \g\big(\boldsymbol{\widehat{\lambda}}\big) \wr S_2$.
\end{theorem}

\begin{proof}
	Let $F^1_\bullet,\dots,F^r_\bullet$ be defining flags for an instance of $\lambda^1,\dots,\lambda^r$ in general position, and let $H \in \Omega_{\boldsymbol{\lambda}}\mathbf{F}$. We let $M_i:= F^i_{m-2}$ for $i=1,2$ and $L_j:= F^j_2$ for $j=3,4$. We begin by showing there is an auxiliary problem $\I^4$ on a $\gr(2,4)$ as a subproblem of $\boldsymbol{\lambda}$. 
	
	Since $\mu^1 \subset \lambda^1,\lambda^2$, $\dim(H \cap M_i) = k-1$. Similarly, since $\mu^2 \subset \lambda^3,\lambda^4$, $\dim(H \cap L_j)= 1$. Let $V=\langle L_3,L_4 \rangle \cong \mathbb{C}^4$ and $H^\prime= H \cap V$. Then we have $H^\prime \in \gr(2,V)$. By construction, $\dim(H^\prime \cap L_j) = 1$. Moreover, since $\codim_H H^\prime = 2$ and $\codim_H M_i \cap H = 1$, we also have $\dim\big(H^\prime \cap (M_i \cap V)\big) = 1$. Thus, the $L_j$ and $M_i \cap V$ each give the condition $\I$ on $\gr(2,V)$. This gives us the Schubert problem $\I^4$ on a $\gr(2,4)$, which has two solutions, $h_1,h_2$, and Galois group $S_2$ by Lemma \ref{lem: aux probs}. 
	
	Let $h_i \in \{h_1,h_2\}$ such that $H^\prime=h_i$. We now show that $H$ may be written as $\langle h_i, h_{i,j} \rangle$ where $h_{i,j}$ satisfies an instance of $\boldsymbol{\widehat{\lambda}}$ on a $\gr(k-2,m-4)$. We consider the space $W=M_1 \cap M_2$. Since $\codim_{\mathbb{C}^m} M_i=2$, we have $\codim_{\mathbb{C}^m} W=4$. Moreover, since $\dim(H \cap M_i) = k-1$, we have $\codim_H \widehat{H} = 2$ where $\widehat{H}=H \cap W$. Therefore, we have $\widehat{H} \in \gr(k-2,W)$. The $M_i$ and $L_j$ are all in general position, thus $V$ and $W$ are in direct sum.
	
	We now want to show that $\widehat{H}$ satisfies an instance of the Schubert problem $\boldsymbol{\widehat{\lambda}}$. Consider the conditions $\lambda^j$ for $j=5,6,\dots,r$, and let $h_iF^j_{\ell}:= \langle F^j_\ell,h_i \rangle \cap W$. Since the $F^j_\ell$ are in general position with $W$, $\dim(F^j_\ell \cap W) = \max\{0,\ell-2\}$. Thus, $\dim(h_iF^j_\ell \cap W)=\ell$. Furthermore, since $\dim(\widehat{H} \cap h_iF^j_\ell) = \dim(H \cap \langle F^j_\ell, h_i \rangle)-2 = \dim(H \cap F^j_\ell)$, we see that $\widehat{H} \in \Omega_{\lambda^j} h_iF^j_\bullet$.

	We now consider the $\lambda^j$ for $j=3,4$. Since $L_j$ and $W$ intersect trivially, the first row of $\lambda^j$ contributes nothing to conditions on $\widehat{H} \in \gr(k-2,W)$.  Furthermore, $L_j \cap H$ is a line of $h_i$, so for $h_iF^j_{\ell-3}:= \langle F^j_\ell, h_i \rangle \cap W$, we have $\dim(h_iF^j_{\ell-3})=\ell-3$ and $\dim(\widehat{H} \cap h_iF^j_{\ell-3})=d-1$. This yields the condition $\overline{\lambda^j-\mu^2}$ on $\gr(k-2,W)$.
	
	Finally, we consider the $\lambda^j$ for $j=1,2$. In this case, $W \subset M_j$, therefore the first column of $\lambda^j$ contributes the trivial condition and may be ignored. If there is some $d<k-1$ such that $\dim(H \cap F^j_\ell) = d$ for some $\ell < m-k+d-1$, then we consider $h_iF^j_{\ell-1}:=\langle h_i, F^j_\ell \rangle \cap W$. Following the same logic as above, we see $\dim(h_iF^j_{\ell-1} \cap \widehat{H}) = d$. Moreover, since $M_1$ and $M_2$ are in general position, $\dim(F^j_\ell \cap W) = \ell-2$ 
	and $\dim(F^j_\ell \cap \widehat{H})=d-1$. Since $\dim(h_iF^j_{\ell-1} \cap \widehat{H}) = d$, we must have $\dim(h_iF^j_{\ell-1}) \geq \ell-1$. Moreover, if $\dim(h_iF^j_{\ell-1}) \geq \dim(F^j_\ell \cap W)+2$, then $h_i$ and $W$ would intersect nontrivially. Since this would contradict $V$ and $W$ being in direct sum, we have $\dim(h_iF^j_{\ell-1}) = \ell-1$. This yields the condition $\overline{\lambda^j-\mu^1}$ on $\gr(k-2,W)$.
	
	We have now shown that $H=\langle h_i, h_{i,j} \rangle$ where $h_i$ is a solution to an instance of $\I^4$ on $\gr(2,V)$ and $h_{i,j}$ is a solution to an instance of $\boldsymbol{\widehat{\lambda}}$ on $\gr(k-2,W)$. Therefore, any permutation in $\g(\boldsymbol{\lambda})$ may be written in the form $(\sigma_1,\sigma_2;\tau)$ with $\sigma_i \in \g\big(\boldsymbol{\hat{\lambda}}\big)$ and $\tau \in \g(\I^4)$ where $(\sigma_1,\sigma_2;\tau) \cdot \langle h_i, h_{i,j} \rangle = \langle h_{\tau(i)}, h_{\tau(i),\sigma_i(j)} \rangle$. Thus, $\g(\boldsymbol{\lambda})$ is a subgroup of $\g\big(\boldsymbol{\widehat{\lambda}}\big) \wr S_2$.
\end{proof}

The second relation involves one condition containing the condition on the right of Figure \ref{fig: relation 2} and three conditions containing the condition on the left of the figure.

\begin{figure}[h]
	\centering
	\begin{minipage}{0.4\textwidth}
		\begin{tabular}{|c|c|c|c|c|}
			\hline
			\cellcolor{blue!25} m-k & m-k-1& \dots & 2 & 1 \\
			\hline
			\cellcolor{blue!25} m-k+1 & m-k & \dots & 3 & 2 \\
			\hline
			\cellcolor{blue!25} \vdots & \vdots & $\ddots$ & \vdots & \vdots \\
			\hline
			\cellcolor{blue!25} m-2 & m-3 & \dots & k & k-1 \\
			\hline
			m-1 & m-2 & \dots & k+1 & k \\
			\hline
		\end{tabular}
	\end{minipage}
	\begin{minipage}{0.05\textwidth}
		~
	\end{minipage}
	\begin{minipage}{0.4\textwidth}
		\begin{tabular}{|c|c|c|c|c|c|}
			\hline
			\cellcolor{blue!25} m-k & \cellcolor{blue!25} m-k-1& \cellcolor{blue!25} \dots & \cellcolor{blue!25} 3 & \cellcolor{blue!25} 2 & 1 \\
			\hline
			\cellcolor{blue!25}m-k+1 & \cellcolor{blue!25}m-k & \cellcolor{blue!25}\dots & \cellcolor{blue!25}4 & 3 & 2 \\
			\hline
			\vdots & \vdots & $\ddots$ & \vdots & \vdots & \vdots \\
			\hline
			m-2 & m-3 & \dots & k+1 & k & k-1 \\
			\hline
			m-1 & m-2 & \dots & k+2 & k+1 & k \\
			\hline
		\end{tabular}
	\end{minipage}
	\caption{The conditions required by the relation in Theorem \ref{thm: relation 2}.}\label{fig: relation 2}
\end{figure}

\begin{theorem}\label{thm: relation 2}
		Let $\boldsymbol{\lambda}=(\lambda^1,\dots,\lambda^r)$ be a reduced Schubert problem on $\gr(k,m)$ with $3 \leq k \leq m-3$. Let $\mu^1=(1,1,\dots,1,0)$ and $\mu^2=(m-k-1,m-k-2,0,\dots,0)$ be conditions on $\gr(k,m)$ and suppose $\mu^1 \subset \lambda^1,\lambda^2, \lambda^3$ and $\mu^2 \subset \lambda^4$. Then, for the Schubert problem
		\[ 
		\boldsymbol{\widehat{\lambda}}=(\overline{\lambda^1-\mu^1},\overline{\lambda^2-\mu^1}, \overline{\lambda^3-\mu^1},\overline{\lambda^4-\mu^2},\lambda^5,\lambda^6,\dots,\lambda^r) 
		\]
		in $\gr(k-2,m-5)$, we have
		$\g(\boldsymbol{\lambda}) \leq \g\big(\boldsymbol{\widehat{\lambda}}\big) \wr S_2$.
\end{theorem}

\begin{proof}
	Let $F^1_\bullet,\dots,F^r_\bullet$ be defining flags for an instance of $\lambda^1,\dots,\lambda^r$ in general position, and let $H \in \Omega_{\boldsymbol{\lambda}}\mathbf{F}$. We let $M_i:= F^i_{m-2}$ for $i=1,2,3$, $L:= F^4_4$, and $L^\prime:=F^4_2$. We begin by showing that there is an auxiliary problem $\I^4$ on a $\gr(2,4)$ as a subproblem of $\boldsymbol{\lambda}$.
	
	We consider the space $\gr(2,L)$ and let $H^\prime=H \cap L \in \gr(2,L)$. Since $L^\prime \subset L$, $\dim(H^\prime \cap L^\prime) = 1$. Furthermore, since the $M_i$ and $L$ are in general position, $\dim(M_i \cap L)=2$ and $\dim\big(H^\prime \cap (M_i \cap L)\big) = 1$. Thus, $L^\prime$ and the $M_i \cap L$ each give the condition $\I$ on $\gr(2,L)$. This gives us the Schubert problem $\I^4$ on a $\gr(2,4)$, which has two solutions, $h_1,h_2$, and Galois group $S_2$ by Lemma \ref{lem: aux probs}.
	
	Let $h_i \in \{h_1,h_2\}$ such that $H^\prime=h_i$. We now show that $H$ may be written as $\langle h_i, h_{i,j} \rangle$ where $h_{i,j}$ satisfies an instance of $\boldsymbol{\widehat{\lambda}}$ on a $\gr(k-2,m-5)$. Let $V:= M_1 \cap M_2$. Since $\dim(H \cap M_i) = k-1$, we consider $\widehat{H}=H\cap V \in \gr(k-2,V)$, which is a $\gr(k-2,m-4)$. The $M_i$ and $L$ are all in general position, thus $V$ and $L$ are in direct sum.
	
	We now want to show that $\widehat{H}$ satisfies an instance of the Schubert problem $\boldsymbol{\widehat{\lambda}}$ on $\gr(k-2,V)$. Consider the conditions $\lambda^j$ for $j=5,6,\dots,r$, and let $h_iF^j_{\ell}:= \langle F^j_\ell,h_i \rangle \cap V$. Since the $F^j_\ell$ are in general position with $V$, $\dim(F^j_\ell \cap V) = \max\{0,\ell-2\}$. Thus, $\dim(h_iF^j_\ell \cap V)=\ell$. Furthermore, since $\dim(\widehat{H} \cap h_iF^j_\ell) = \dim(H \cap \langle F^j_\ell, h_i \rangle)-2 = \dim(H \cap F^j_\ell)$, we see that $\widehat{H} \in \Omega_{\lambda^j} h_iF^j_\bullet$.
	
	Next, we consider $\lambda^4$. Since $V$ and $L$ have trivial intersection, $\dim(F^4_\ell \cap V)=0$ for $\ell \leq 4$. Moreover, since $F^4_4 \cap H=H^\prime$, this implies that $\dim(F^4_\ell \cap \widehat{H})=\max\{0,\dim(F^4_\ell \cap H)-2\}$ for all $\ell$. This gives us the condition $\overline{\lambda^4-\mu^2}$ on the flag $F^4_5 \cap V \subset \dots \subset F^4_m \cap V$.
	
	We now consider the $\lambda^j$ for $j=1,2$. In this case, $V \subset M_j$, therefore the first column of $\lambda^j$ contributes the trivial condition and may be ignored. If there is some $d<k-1$ such that $\dim(H \cap F^j_\ell) = d$ for some $\ell < m-k+d-1$, then we consider $h_iF^j_{\ell-1}:=\langle h_i, F^j_\ell \rangle \cap V$. We have 
	\[
	\dim(h_iF^j_{\ell-1} \cap \widehat{H}) = \dim(\langle h_i, F^j_\ell \rangle \cap H) - 2 = d.
	\] 
	Moreover, since $M_1$ and $M_2$ are in general position, $\dim(F^j_\ell \cap V) = \ell-2$ 
	and $\dim(F^j_\ell \cap \widehat{H})=d-1$. Since $\dim(h_iF^j_{\ell-1} \cap \widehat{H}) = d$, we must have $\dim(h_iF^j_{\ell-1}) \geq \ell-1$. Moreover, if $\dim(h_iF^j_{\ell-1}) \geq \dim(F^j_\ell \cap V)+2$, then $h_i$ and $V$ would intersect nontrivially. Since this would contradict $V$ and $L$ being in direct sum, we have $\dim(h_iF^j_{\ell-1}) = \ell-1$. This yields the condition $\overline{\lambda^j-\mu^1}$ on $\gr(k-2,V)$.

	Finally, we consider $\lambda^3$. Let $\dim(H \cap F^3_\ell) \geq d_\ell$. For each $h_i$, let $h_iF^3_\ell= \langle h_i, F^3_{\ell+2} \rangle$. Since $\dim(H \cap h_iF^3_\ell) \geq d_{\ell+2}+2$ and $\codim_H \widehat{H}=2$, we must have $\dim(\widehat{H} \cap h_iF^3_\ell) \geq d_{\ell+2}$. Furthermore, since the flags are all in general position with respect to one another, we have $\dim(h_iF^3_\ell)=\ell$. Thus, $\widehat{H}$ satisfies the condition $(\lambda^3_1,\dots,\lambda^3_{k-2})$ on the flag $h_iF^3_\bullet$.
	
	We now have $\widehat{H}$ satisfies an instance of $(\overline{\lambda^1-\mu^1},\overline{\lambda^2-\mu^1}, \lambda^3,\overline{\lambda^4-\mu^2},\lambda^5,\lambda^6,\dots,\lambda^r)$ on a $\gr(2,m-4)$. However, since $\lambda^3_{k-2}=1$, by Lemma \ref{lem: deficient}, this reduces to an instance of the Schubert problem $\boldsymbol{\hat{\lambda}}$ on a $\gr(2,m-5)$.
	
	We have now shown that $H=\langle h_i, h_{i,j} \rangle$ where $h_i$ is a solution to an instance of $\I^4$ on $\gr(2,L)$ and $h_{i,j}$ is a solution to an instance of $\boldsymbol{\widehat{\lambda}}$ on $\gr(k-2,V)$. Therefore, any permutation in $\g(\boldsymbol{\lambda})$ may be written in the form $(\sigma_1,\sigma_2;\tau)$ with $\sigma_i \in \g\big(\boldsymbol{\hat{\lambda}}\big)$ and $\tau \in \g(\I^4)$ where $(\sigma_1,\sigma_2;\tau) \cdot \langle h_i, h_{i,j} \rangle = \langle h_{\tau(i)}, h_{\tau(i),\sigma_i(j)} \rangle$. Thus, $\g(\boldsymbol{\lambda})$ is a subgroup of $\g\big(\boldsymbol{\widehat{\lambda}}\big) \wr S_2$.
\end{proof}

The third relation involves five conditions: one containing each of the conditions in the top row of Figure \ref{fig: relation 3} and three containing the condition in the bottom row.

\begin{figure}[h]
	\centering
	\begin{minipage}{0.4\textwidth}
		\begin{tabular}{|c|c|c|c|c|}
			\hline
			\cellcolor{blue!25} m-k &  \cellcolor{blue!25} \dots & \cellcolor{blue!25} 3 & 2 & 1 \\
			\hline
			 m-k+1 & \dots & 4 & 3 & 2 \\
			\hline
			\vdots & $\ddots$ & \vdots & \vdots & \vdots \\
			\hline
			m-1 & \dots & k+2 & k+1 & k \\
			\hline
		\end{tabular}
	\end{minipage}
	\begin{minipage}{0.05\textwidth}
		~
	\end{minipage}
	\begin{minipage}{0.4\textwidth}
		\begin{tabular}{|c|c|c|c|}
			\hline
			\cellcolor{blue!25} m-k & \cellcolor{blue!25} \dots & \cellcolor{blue!25} 2  & 1 \\
			\hline
			m-k+1 & \dots & 3 & 2 \\
			\hline
			\vdots & $\ddots$ & \vdots & \vdots \\
			\hline
			m-1 & \dots & k+1 & k \\
			\hline
		\end{tabular}
	\end{minipage}
	\vspace{0.2cm} ~ \\
	
		\begin{tabular}{|c|c|c|c|}
			\hline
			\cellcolor{blue!25} m-k & m-k-1& \dots & 1 \\
			\hline
			\cellcolor{blue!25} \vdots & \vdots & $\ddots$ & \vdots \\
			\hline
			\cellcolor{blue!25} m-2 & m-3 & \dots & k-1 \\
			\hline
			m-1 & m-2 & \dots  & k \\
			\hline
		\end{tabular}
	\caption{The conditions required by the relation in Theorem \ref{thm: relation 3}.}\label{fig: relation 3}
\end{figure}

\begin{theorem}\label{thm: relation 3}
	Let $\boldsymbol{\lambda}=(\lambda^1,\dots,\lambda^r)$ be a reduced Schubert problem on $\gr(k,m)$ with $3 \leq k \leq m-4$. Let $\mu^1=(m-k-1,0,0,\dots,0)$, $\mu^2=(m-k-2,0,\dots,0)$, and $\mu^3=(1,1,\dots,1,0)$  be conditions on $\gr(k,m)$ and suppose $\mu^1 \subset \lambda^1$, $\mu^2 \subset \lambda^2$, and $\mu^3 \subset \lambda^3,\lambda^4,\lambda^5$. Then, for the Schubert problem
	\[ 
	\boldsymbol{\widehat{\lambda}}=(\overline{\lambda^1-\mu^1},\overline{\lambda^2-\mu^2}, \overline{\lambda^3-\mu^3}, \overline{\lambda^4-\mu^3}, \overline{\lambda^5-\mu^3},  \lambda^6,\lambda^7,\dots,\lambda^r) 
	\]
	in $\gr(k-2,m-5)$, we have
	$\g(\boldsymbol{\lambda}) \leq \g\big(\boldsymbol{\widehat{\lambda}}\big) \wr S_3$.
\end{theorem}

\begin{proof}
Let $F^1_\bullet,\dots,F^r_\bullet$ be the defining flags of an instance of $\boldsymbol{\lambda}$ in general position, and let $L:= F^1_2$, $K:= F^2_3$, and $J_i=F^i_{m-2}$ for $i=3,4,5$. We begin by showing that there is an auxiliary problem $\T\cdot\I^4$ on a $\gr(2,5)$ as a subproblem of $\boldsymbol{\lambda}$. 

Let $V:=\langle L,K \rangle \cong \mathbb{C}^5$ and $H^\prime:= H \cap V$ for $H \in \Omega_{\boldsymbol{\lambda}}\mathbf{F}$. Since $\dim(H \cap L) = 1$ and $\dim(H \cap K) \geq 1$, we have $H^\prime \in \gr(2,V)$. For each $J_i$, we have $\codim_{\mathbb{C}^m} J_i = 2$ and $\dim(H \cap J_i) = 3$. Thus, $\dim(J_i \cap V) = 3$ and $\dim(H^\prime \cap J_i) = 1$. Hence, the $J_i$ and $K$ each give the condition $\I$ on $\gr(2,V)$ while $L$ gives the condition $\T$ on $\gr(2,V)$. This gives us the Schubert problem $\T \cdot \I^4$ on a $\gr(2,5)$, which has three solutions $h_1,h_2,h_3$ and Galois group $S_3$ by Lemma \ref{lem: aux probs}.

Let $h_i \in \{h_1,h_2,h_3\}$ such that $H^\prime=h_i$. We now show that $H$ may be written as $\langle h_i, h_{i,j} \rangle$ where $h_{i,j}$ satisfies an instance of $\boldsymbol{\widehat{\lambda}}$ on a $\gr(k-2,m-5)$. Let $W:=J_3 \cap J_4 \cap \langle h_i, J_5 \rangle$. Since $\codim_{\mathbb{C}^m}J_i=2$, $\codim_H J_i=1$, and $\dim(h_i \cap J_3)=1$, we see that $\dim(W)=m-5$. Let $\widehat{H}=H \cap W$. Moreover, since $L$, $K$, and the $J_i$ are in general position, $V$ and $W$ are in direct sum.

We now want to show that $\widehat{H}$ satisfies an instance of the Schubert problem $\boldsymbol{\widehat{\lambda}}$ on $\gr(k-2,W)$. Consider the conditions $\lambda^j$ for $j=6,7,\dots,r$, and let $h_iF^j_{\ell}:= \langle F^j_\ell,h_i \rangle \cap W$. Since the $F^j_\ell$ are in general position with $W$, $\dim(F^j_\ell \cap W) = \max\{0,\ell-2\}$. Thus, $\dim(h_iF^j_\ell \cap W)=\ell$. Furthermore, since $\dim(\widehat{H} \cap h_iF^j_\ell) = \dim(H \cap \langle F^j_\ell, h_i \rangle)-2 = \dim(H \cap F^j_\ell)$, we see that $\widehat{H} \in \Omega_{\lambda^j} h_iF^j_\bullet$.

We next consider what conditions on $\widehat{H} \in \gr(k-2,W)$ are implied by the condition $\lambda^1$ on $\gr(k,m)$. Let $h_iF^1_\ell=\langle h_i, F^1_{\ell+5}\rangle \cap W$. Since $\dim(h_i \cap F^1_4)=1$ and $W$ is in general position with $F^1_\bullet$, we have $\dim(h_iF^1_\ell)=\ell$. Moreover, if $\dim(H \cap F^1_{\ell+5})=d$, then $\dim(H \cap \langle h_i, F^1_{\ell+5})=d+1$. In this case, we see $\dim\big(\widehat{H} \cap h_iF^1_\ell)=d-1$. This gives the condition $\overline{\lambda^1-\mu^1}$ on the flag $h_iF^1_1\subset\dots\subset h_iF^1_{m-5}$. Repeating the same argument for $\lambda^2$ and $F^2_\bullet$ with the exception of defining $h_iF^2_\ell=\langle h_i, F^1_{\ell+4}\rangle \cap W$ gives the condition $\overline{\lambda^2-\mu^2}$ on the flag $h_iF^2_1\subset\dots\subset h_iF^2_{m-5}$.

Finally, we consider what conditions on $\gr(k-2,W)$ are implied by $\lambda^j$ on $\gr(k,m)$ for $j=3,4,5$. Let $h_iF^j_{\ell-1}:=\langle h_i, F^j_\ell \rangle \cap W$. Since our flags are in general position, $\dim(F^j_\ell \cap W) = \ell-2$ 
and $\dim(F^j_\ell \cap \widehat{H})=d-1$. Since $\dim(h_iF^j_{\ell-1} \cap \widehat{H}) = d$, we must have $\dim(h_iF^j_{\ell-1}) \geq \ell-1$. Moreover, if $\dim(h_iF^j_{\ell-1}) \geq \dim(F^j_\ell \cap W)+2$, then $h_i$ and $W$ would intersect nontrivially. Since this would contradict $V$ and $W$ being in direct sum, we have $\dim(h_iF^j_{\ell-1}) = \ell-1$. Moreover, if there is some $d<k-1$ such that $\dim(H \cap F^j_\ell) = d$ for some $\ell < m-k+d-1$, then
\[
\dim(h_iF^j_{\ell-1} \cap \widehat{H}) = \dim(\langle h_i, F^j_\ell \rangle \cap H) - 2 = d.
\] 
This yields the condition $\overline{\lambda^j-\mu^1}$ on $\gr(k-2,W)$.

We have now shown that $H=\langle h_i, h_{i,j} \rangle$ where $h_i$ is a solution to an instance of $\T\cdot\I^4$ on $\gr(2,V)$ and $h_{i,j}$ is a solution to an instance of $\boldsymbol{\widehat{\lambda}}$ on $\gr(k-2,W)$. Therefore, any permutation in $\g(\boldsymbol{\lambda})$ may be written in the form $(\sigma_1,\sigma_2,\sigma_3;\tau)$ with $\sigma_i \in \g\big(\boldsymbol{\hat{\lambda}}\big)$ and $\tau \in \g(\T\cdot\I^4)$ where $(\sigma_1,\sigma_2,\sigma_3;\tau) \cdot \langle h_i, h_{i,j} \rangle = \langle h_{\tau(i)}, h_{\tau(i),\sigma_i(j)} \rangle$. Thus, $\g(\boldsymbol{\lambda})$ is a subgroup of $\g\big(\boldsymbol{\widehat{\lambda}}\big) \wr S_3$.
\end{proof}

For the remainder of the relations, we will restrict our focus specifically to problems on $\gr(4,9)$. We will continue with relations that rely on only a few key conditions being present in the problem to break it down into smaller problems and determine restrictions on the Galois group from there.

\begin{theorem}\label{thm: relation 4}
	Let $\boldsymbol{\lambda}=(\lambda^1, \dots,\lambda^r)$ be a reduced Schubert problem on $\gr(4,9)$ with $\lambda^1 = \ThTh$, $\lambda^2 = \TT$, $\lambda^3 \supset \F$, and $\lambda^4 \supset \III$. Then, for the Schubert problem
	\[ 
	\boldsymbol{\widehat{\lambda}}=(\T, \I, \overline{\lambda^3-\F},\overline{\lambda^4-\III}, \lambda^5,\dots,\lambda^r) 
	\]
	in $\gr(2,4)$, we have
	$\g(\boldsymbol{\lambda}) \leq \g\big(\boldsymbol{\widehat{\lambda}}\big) \wr S_2$.
\end{theorem}

\begin{proof}
	Let $F^1_\bullet, \dots, F^r_\bullet$ be the defining flags of an instance of $\boldsymbol{\lambda}$ in general position, and let $M:=F^1_4$, $K:=F^2_5$, $L:=F^3_2$, and $J:=F^4_7$. Consider $V:= \langle L, M \rangle \cap \langle L, K \rangle$ and $H^\prime := H \cap V$ where $H \in \Omega_{\boldsymbol{\lambda}}\mathbf{F}$. Since $\dim \langle L,M \rangle = 6$ and $\dim \langle L, K \rangle =7$, $\dim V \geq 4$. Moreover, if $\dim V > 4$, then $\dim \big(M \cap \langle L , K \rangle\big) \geq 3$, but these spaces are in general position. Thus, $\dim V = 4$. Moreover, since $\dim \big(H \cap \langle L,M \rangle\big) = 3$ and $\dim \big(H \cap \langle L, K \rangle\big) =3$ and these spaces are in general position, $\dim H^\prime =2$. Therefore, we have $H^\prime \in \gr(2,V)$.
	
	Since $L \subset V$, $L$ gives the condition $\I$ on $\gr(2,V)$. Moreover, $\dim \big(\langle L, M \rangle \cap V \big) = 4$ and $L \subset V$, thus $\dim (M \cap V) = 2$. Since $H^\prime \subset \langle L, M \rangle$, $\dim (H^\prime \cap L) = 1$, and $\dim \big(H \cap \langle L,M \rangle ) = 3$, we have $\dim (H^\prime \cap M)= 1$ as well. Thus, $M \cap V$ gives the condition $\I$ on $\gr(2,V)$. Similarly, $K \cap V$ gives the condition $\I$ on $\gr(2,N)$. Finally, $\dim(J \cap H^\prime)=1$ since $\dim (J \cap H)=3$, and $\dim(J \cap V)=4-2=2$ since $\codim_{\mathbb{C}^9} J=2$. Thus, $J \cap V$ gives the condition $\I$ on $\gr(2,V)$. This gives us the auxiliary problem $\I^4$ on a $\gr(2,4)$ as a subproblem of $\boldsymbol{\lambda}$. This auxiliary problem has two solutions $h_1,h_2$ and Galois group $S_2$ by Lemma \ref{lem: aux probs}.
	
	Suppose $H \supset H^\prime=h_i$ where $h_i \in \{h_1,h_2\}$. We consider $\widehat{V}:=\langle J \cap K, J \cap M \rangle$. Since $\dim(J \cap K)= 3$ and $\dim(J \cap M)=2$, we see $\dim(\widehat{V})=5$. Let $\widehat{H}:= H \cap \widehat{V}$. Then $\dim(\widehat{H}) \geq 2$ since $\dim\big(H \cap (J \cap K)\big), \dim\big(H \cap (J \cap M)\big) \geq 1$. Moreover, since $J,K,M,$ and $L$ are all in general position with respect to one another, we have $\mathbb{C}^9 = V \oplus \widehat{V}$. Using arguments analogous to those in the proof of Theorem \ref{thm: relation 1}, we know $\widehat{H}$ satisfies the Schubert conditions $\overline{\lambda^3-\F}$ and $\overline{\lambda^4-\III}$ on the flags $F^3_\bullet/V$ and $F^4_\bullet/V$ respectively as well as the conditions $\lambda^5, \dots, \lambda^r$ on the flags $h_iF^5_\bullet,\dots,h_iF^r_\bullet$ respectively.

	We first look to the relation $\widehat{H}$ has with $M$. Since $\widehat{V}:=\langle J \cap K, J \cap M \rangle$, $\dim(J \cap M) = 2$, and $\dim\big(H \cap (J \cap M)\big) = 1$, $M \cap \widehat{V}$ gives the condition $\T$. The same argument using $K$ gives us the condition $\I$. Thus, we have that the $\widehat{H}$ are solutions to the Schubert problem $\boldsymbol{\widehat{\lambda}}$. Thus, by the same arguments as in the proof of Theorem \ref{thm: relation 1}, we see the Galois group of $\boldsymbol{\lambda}$ is a subgroup of $\g\big(\boldsymbol{\widehat{\lambda}}\big) \wr S_2$.
\end{proof}

\begin{theorem}\label{thm: relation 7}
	The Galois groups of the Schubert problems $\ThTh \cdot \ThT \cdot \F \cdot \III \cdot \I^2, \ThTh \cdot \ThT \cdot \F \cdot \TII \cdot I, \ThTh \cdot \ThT \cdot \FI \cdot \III \cdot \I$, $\ThTh \cdot \ThT \cdot \FI \cdot \TI$, 
	$\ThThI \cdot \F \cdot \TT \cdot \III \cdot \I^2, \ThThI \cdot \F \cdot \TT \cdot \TII \cdot \I, \ThThI \cdot \FI \cdot \TT \cdot \III \cdot \I$, and $\ThThI \cdot \FI \cdot \TT \cdot \TII$
	on $\gr(4,9)$ are subgroups of $S_2 \wr S_2$.
\end{theorem}

\begin{proof}[Sketch of proof.]
	The proof of this theorem is analogous to that of Theorem \ref{thm: relation 4} with the following changes:
	
	When considering the Schubert problem $\ThTh \cdot \ThT \cdot \F \cdot \III \cdot \I^2, \ThTh \cdot \ThT \cdot \F \cdot \TII \cdot I, \ThTh \cdot \ThT \cdot \FI \cdot \III \cdot \I$, or $\ThTh \cdot \ThT \cdot \FI \cdot \TI$, we let $K^\prime := \langle h_i, F^2_3 \rangle$. We then define $V:=\langle J \cap K^\prime, J \cap M \rangle$ as a $\gr(2,4)$ and continue the proof as above replacing instances of $K$ with $K^\prime$.
	
	When considering the Schubert problem $\ThThI \cdot \F \cdot \TT \cdot \III \cdot \I^2, \ThThI \cdot \F \cdot \TT \cdot \TII \cdot \I, \ThThI \cdot \FI \cdot \TT \cdot \III \cdot \I$, or $\ThThI \cdot \FI \cdot \TT \cdot \TII$, we let $M^\prime:= \langle h_i, F^1_7 \rangle$. We then define $V:= \langle J \cap K, J \cap M \rangle \cap M^\prime$ as a $\gr(2,4)$ and continue the proof as above replacing instances of $M$ with $M^\prime$.
\end{proof}

\begin{theorem}\label{thm: relation 5}
	The Galois groups of the Schubert problems $\FTT \cdot \TT \cdot \III^2 \cdot \I^2,\FTT \cdot \TT \cdot \TII \cdot \III \cdot \I$, and $\FTT \cdot \TT \cdot \TII^2$ on $\gr(4,9)$ are subgroups of $S_2 \wr S_2$.
\end{theorem}

\begin{proof}
	We first show that the  Galois group of $\boldsymbol{\lambda}=\FTT \cdot \TT \cdot \III^2 \cdot \I^2$ is a subgroup of $S_2 \wr S_2$. Let $F^1_\bullet, \dots, F^6_\bullet$ be the defining flags of $\boldsymbol{\lambda}$ in general position, and let $M_2:=F^1_2, M_6:=F^1_6, K:=F^2_5, J:=F^3_7, J^\prime:=F^4_7, I:=F^5_5,$ and $I^\prime:=F^6_5$. We begin by showing there is an auxiliary problem $\I^4$ on a $\gr(2,4)$ as a subproblem of $\boldsymbol{\lambda}$.
	
	Let $V:=\langle M_6 \cap K, M_2 \rangle \cong \mathbb{C}^4$ and $H^\prime:= H \cap V$ for $H \in \Omega_{\boldsymbol{\lambda}} \mathbf{F}$. Since $\dim(H \cap M_6)=3$ and $\dim(H \cap K)=2$, we have $\dim(H \cap M_6 \cap K)=1$. Thus, $H^\prime \in \gr(2,V)$ and $K \cap V$ gives us the condition $\I$ on $\gr(2,V)$. Since $M_2 \subset V$, $M_2$ gives us the condition $\I$ on $\gr(2,V)$. Moreover, $\dim(J \cap V)=2$ and $\dim(H^\prime \cap J)=1$, thus $J$ gives us the condition $\I$ on $\gr(2,V)$. Similarly, $J^\prime$ gives the condition $\I$ as well. This gives us the Schubert problem $\I^4$ on a $\gr(2,4)$, which has two solutions $h_1, h_2$ and Galois group $S_2$ by Lemma \ref{lem: aux probs}.
	
	Let $h_i \in \{h_1,h_2\}$ such that $H^\prime = h_i$. We now consider the space $\widehat{V}:= J \cap J^\prime \cong \mathbb{C}^5$. Since our flags are in general position, we have $V \oplus \widehat{V} \cong \mathbb{C}^9$. Let $\widehat{H}=H \cap \widehat{V}$. Since $\dim(H \cap J)=\dim(H \cap J^\prime)=3$, $\dim \widehat{H}=2$. Thus, $\widehat{H} \in \gr(2,\widehat{V})$. Moreover, since $M_6$ and $\widehat{V}$ are in general position and $\dim(H \cap M_6)=3$, $\dim(\widehat{H} \cap M_6)=1$. This gives the condition $\T$ on $\gr(2,\widehat{V})$. Similarly, $\langle K, h_i \rangle$ gives the condition $\T$ on $\gr(2,\widehat{V})$. Finally, $\dim \big(\langle I, h_i \rangle \cap \widehat{V}\big)= \dim \big(\langle I^\prime, h_i \rangle \cap \widehat{V}\big)=3$. Thus, $\langle I,h_i \rangle \cap \widehat{V}$ and $\langle I^\prime,h_i \rangle \cap \widehat{V}$ both give the condition $\I$ on $\gr(2,\widehat{V})$. Thus, we have the Schubert problem $\I^2\cdot \T^2$ on a $\gr(2,5)$, which has two solutions and Galois group $S_2$ by Lemma \ref{lem: aux probs}. Thus, by the same argument as in the previous proofs, $\g(\boldsymbol{\lambda})$ is a subgroup of $S_2 \wr S_2$.
	
	For each of the other Schubert problems, the argument is the same replacing any missing $\I$ conditions with the new $\overline{\TII-\III}$ condition in the same manner as in the proof of Theorem \ref{thm: relation 1}.
\end{proof}

\begin{theorem}\label{thm: relation 6}
	The Galois groups of the Schubert problems $\ThThT \cdot \TT^2 \cdot \III \cdot \I$ and $\ThThT \cdot \TT^2 \cdot \TII$ on $\gr(4,9)$ are subgroups of $S_2 \wr S_2$.
\end{theorem}

\begin{proof}
	We first prove the Galois group of $\boldsymbol{\lambda}=\ThThT \cdot \TT^2 \cdot \III \cdot \I$ is a subgroup of $S_2 \wr S_2$. Let $F^1_\bullet, \dots, F^5_\bullet$ be defining flags for $\boldsymbol{\lambda}$ in general position, and let $L_4:= F^1_4, L_6:=F^1_6, K:=F^2_5, K^\prime := F^3_5, J:= F^4_7,$ and $I:= F^5_5$. We begin by showing there is an auxiliary problem $\I^4$ on a $\gr(2,4)$ as a subproblem of $\boldsymbol{\lambda}$.
	
	Let $V:=\langle L_6 \cap K, L_6 \cap K^\prime \rangle \cong \mathbb{C}^4$ and $H^\prime:= H \cap V$ for $H \in \Omega_{\boldsymbol{\lambda}} \mathbf{F}$. Since $\dim(H \cap L_6)=3$ and $\dim(H \cap K)=\dim(H \cap K^\prime)=2$, we have $\dim H^\prime = 2$. Thus, $H^\prime \in \gr(2,V)$ and both $K \cap V$ and $K^\prime \cap V$ give us the condition $\I$ on $\gr(2,V)$. Moreover, since $V \subset L_6$ and $\codim_{L_6} L_4 = 2$, $L_4$ gives us the condition $\I$ on $\gr(2,V)$ as well. Finally, $\dim(J \cap V)=2$ and $\dim(H^\prime \cap J)=1$, thus $J$ gives us the condition $\I$ on $\gr(2,V)$. We have the Schubert problem $\I^4$ on a $\gr(2,4)$, which has two solutions $h_1, h_2$ and Galois group $S_2$ by Lemma \ref{lem: aux probs}.
	
	Let $h_i \in \{h_1, h_2 \}$ such that $H^\prime = h_i$. We now consider the space $\widehat{V}:=\langle J \cap L_4, J \cap K \rangle \cong \mathbb{C}^5$. Since our flags are in general position, we have $V \oplus \widehat{V} \cong \mathbb{C}^9$. Let $\widehat{H}=H \cap \widehat{V}$. Since $\dim(H \cap J)=3$ and $\dim(H \cap L_4) = \dim(H \cap K)=2$, $\dim \widehat{H}=2$. Thus, $\widehat{H} \in \gr(2,\widehat{V})$. Thus, $\widehat{H} \in \gr(2,\widehat{V})$. Moreover, $\dim(J \cap L_4)=2$ and $\dim(J \cap K)=3$, thus $L_4 \cap \widehat{V}$ and $K \cap \widehat{V}$ give conditions $\T$ and $\I$ respectively on $\gr(2,\widehat{V})$. Since $\dim(h_i \cap K^\prime)=1$, $\dim \langle h_i, K^\prime\rangle =6$ and $\dim \big( H \cap  \langle h_i, K^\prime\rangle \big)=3$, thus $\langle h_i, K^\prime\rangle \cap \widehat{V}$ gives the condition $\T$ on $\gr(2,\widehat{V})$. Finally, $\dim \big( \langle h_i, I \rangle \cap \widehat{V} \big) = 3$, so $\langle h_i, I \rangle \cap \widehat{V}$ gives the condition $\I$ on $\gr(2,\widehat{V})$. Thus, by the same argument as in the previous proofs, $\g(\boldsymbol{\lambda})$ is a subgroup of $S_2 \wr S_2$.
	
	For $\ThThT \cdot \TT^2 \cdot \TII$, the argument is the same replacing the missing $\I$ condition with the new $\overline{\TII-\III}$ condition in the same manner as in the proof of Theorem \ref{thm: relation 1}.
\end{proof}

\begin{theorem}\label{thm: relation 9}
	The Galois groups of the Schubert problems $\I^2 \cdot \III^2 \cdot \TT^2 \cdot \F, \I \cdot \III \cdot \TII \cdot \TT^2 \cdot \F, \I \cdot \III^2 \cdot \TT^2 \cdot \FI, \TII^2 \cdot \TT^2 \cdot \F,$ and $\III \cdot \TII \cdot \TT^2 \cdot \FI $ on $\gr(4,9)$ are subgroups of $S_2 \wr S_3$.
\end{theorem}

\begin{proof}
	We first prove the Galois group of $\boldsymbol{\lambda}=\I^2 \cdot \III^2 \cdot \TT^2 \cdot \F$ is a subgroup of $S_2 \wr S_3$. Let $F^1_\bullet, \dots, F^7_\bullet$ be defining flags for $\boldsymbol{\lambda}$ in general position, and let $I:=F^1_5, I^\prime:=F^2_5, J:=F^3_7, J^\prime:=F^4_7, K:=F^5_5, K^\prime:=F^6_5, L:=F^7_2$. We begin by showing there is an auxiliary problem $\I^4 \cdot \T$ on a $\gr(2,5)$ as a subproblem of $\boldsymbol{\lambda}$.
	
	Let $V:= \langle L, K \rangle \cap \langle L, K^\prime \rangle \cong \mathbb{C}^5$ and $H^\prime := H \cap V$ for $H \in \Omega_{\boldsymbol{\lambda}}\mathbf{F}$. Since $\dim(H \cap L)=1$ and $\dim(H \cap K)=\dim(H \cap K^\prime)=2$, we have $\dim H^\prime = 2$. Thus, $H^\prime \in \gr(2,V)$. Since $L \subset V$, $L$ gives the condition $\T$ on $\gr(2,V)$. Moreover, since $\dim \big(K \cap \langle L, K^\prime \rangle \big)=3$, $K \cap V$ gives the condition $\I$ on $\gr(2,V)$. Similarly, $K^\prime \cap V$ gives the condition $\I$. Finally, $\dim(J \cap V)=\dim(J^\prime \cap V)=3$, thus $J \cap V$ and $J^\prime \cap V$ both give the condition $\I$ on $\gr(2,V)$. We have the Schubert problem $\I^4 \cdot \T^2$ on a $\gr(2,5)$, which has three solutions $h_1,h_2,h_3$ and Galois group $S_3$ by Lemma \ref{lem: aux probs}.
	
	Let $h_i \in \{h_1,h_2,h_3\}$ such that $H^\prime = h_i$. Let $N=\langle h_i,K \rangle \cap J \cap J^\prime$ and $N^\prime = \langle h_i, K^\prime \rangle \cap J \cap J^\prime$. Then $\widehat{V}=\langle N, N^\prime \cong \mathbb{C}^4$. Since our flags are in general position, $V \oplus \widehat{V} \cong \mathbb{C}^9$. Let $\widehat{H}=H \cap \widehat{V}$. Since $\dim N = \dim N^\prime = 2$, $\langle h_i,K \rangle \cap \widehat{V}$ and $\langle h_i, K^\prime \rangle \cap \widehat{V}$ both give the condition $\I$ on $\gr(2,\widehat{V})$. Moreover, $\langle h_i, I \rangle$ and $\langle h_i, I^\prime \rangle$ both give the condition $\I$ on $\gr(2,\widehat{V})$ as well. Thus, by the same argument as in the previous proofs, $\g(\boldsymbol{\lambda})$ is a subgroup of $S_2 \wr S_3$.
	
	For the other Schubert problems, the argument is the same replacing the missing $\I$ condition with the new $\overline{\TII-\III}$ or $\overline{\FI - \F}$ condition in the same manner as in the proof of Theorem \ref{thm: relation 1}.
\end{proof}

\chapter{ELEVEN TYPES OF SCHUBERT PROBLEMS}

Here, we list all $148$ deficient problems in $\gr(4,9)$. We will organize the problems based on the relations in Chapter 3, and we will determine the Galois group for each of these types of Schubert problem. Additionally, we will include a table showing the results of sampling from the Galois group of the problems using Algorithm \ref{alg:Frob} for each Galois group found. On these tables, the column ``Empirical Fraction'' gives the quantity (order of Galois group)$\cdot$(number of times cycle type observed)/(number of cycle types sampled).

\section{Type one}

Here, we study $\I^2 \cdot \T^2 \cdot \III^2 \cdot \F^2=4$ and all conclusions about this problem will also apply to the problems in Figure \ref{fig: type 1}. By Theorem \ref{thm: relation 1}, the Galois group of each of these problems is a subgroup of either $\g(\I^4 \cdot \II) \wr S_2$, $\g(\I^2 \cdot \T^2) \wr S_2$, or $\g(\I^3 \cdot \TI) \wr S_2$. However, in $\gr(2,5)$, all of these reduce to the problem $\I^4$ in $\gr(2,4)$ by Lemma \ref{lem: deficient}. By Lemma \ref{lem: aux probs}, this has Galois group $S_2$, thus we see that the Galois group of each of these problems is a subgroup of $S_2 \wr S_2 \cong D_4$. To see that this is the Galois group, we use Algorithm \ref{alg:Frob} yielding results similar to Table \ref{table: D4}.

\begin{table}
\centering
\caption{Example of Algorithm \ref{alg:Frob} output when the Galois group is $S_2 \wr S_2$.}\label{table: D4}
\begin{tabular}{|c|c|c|c|} 
	\hline 
	\multicolumn{4}{|c|}{\rule{0pt}{12pt}Cycles in $\g(\ThThT \cdot \TT^2 \cdot \TII)$} \\
	\multicolumn{4}{|c|}{found in 49606 samples} \\
	\hline
	Cycle & Observed & Empirical & Number  \\
	Type & Frequency & Fraction & in $S_2 \wr S_2$ \\
	\hline
	(4) &  12645 & 2.0392 & 2 \\
	\hline
	(2,2) &  18520 & 2.9867 & 3 \\
	\hline
	(2,1,1) &  12292 & 1.9823 & 2 \\
	\hline
	(1,1,1,1) &  6149 & 0.9916 & 1 \\
	\hline
\end{tabular}
\end{table}

\begin{figure}[h]
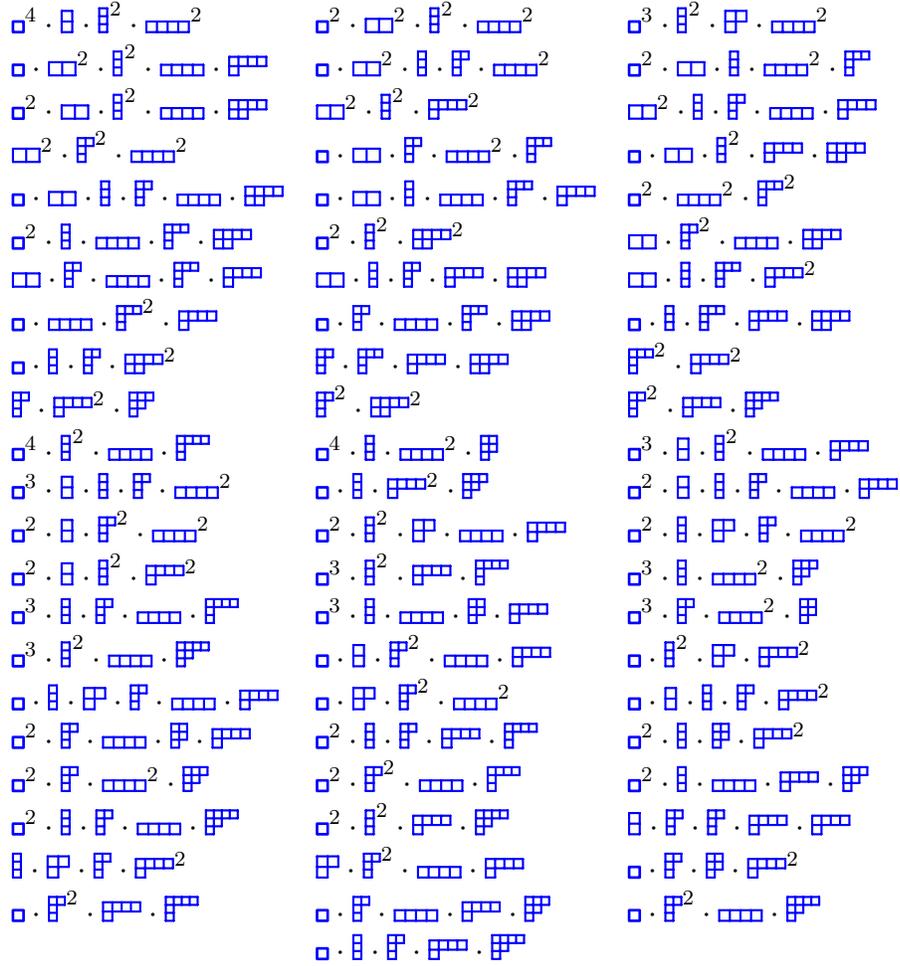

\centering
\begin{tabular}{lll}
	$\I^4 \cdot \II \cdot \III^2 \cdot \F^2$ & $\I^2 \cdot \T^2 \cdot \III^2 \cdot \F^2$ & 
	$\I^3 \cdot \III^2 \cdot \TI \cdot \F^2$ \\
	$\I \cdot \T^2 \cdot \III^2 \cdot \F \cdot \FI$ &%
	
	$\I \cdot \T^2 \cdot \III \cdot \TII \cdot \F^2$ &%
	
	$\I^2 \cdot \T \cdot \III \cdot \F^2 \cdot \ThII$ \\
	
	$\I^2 \cdot \T \cdot \III^2 \cdot \F \cdot \FT$ &%
	
	$\T^2 \cdot \III^2 \cdot \FI^2$ &%
	
	$\T^2 \cdot \III \cdot \TII \cdot \F \cdot \FI$ \\
	
	$\T^2 \cdot \TII^2 \cdot \F^2$ &%
	
	$\I \cdot \T \cdot \TII \cdot \F^2 \cdot \ThII$ &%
	
	$\I \cdot \T \cdot \III^2 \cdot \FI \cdot \FT$ \\
	
	$\I \cdot \T \cdot \III \cdot \TII \cdot \F \cdot \FT$ &%
	
	$\I \cdot \T \cdot \III \cdot \F \cdot \ThII \cdot \FI$ &%
	
	$\I^2 \cdot \F^2 \cdot \ThII^2$ \\
	
	$\I^2 \cdot \III \cdot \F \cdot \ThII \cdot \FT$ &%
	
	$\I^2 \cdot \III^2 \cdot \FT^2$ &%
	
	$\T \cdot \TII^2 \cdot \F \cdot \FT$ \\
	
	$\T \cdot \TII \cdot \F \cdot \ThII \cdot \FI$ &%
	
	$\T \cdot \III \cdot \TII \cdot \FI \cdot \FT$ &%
	
	$\T \cdot \III \cdot \ThII \cdot \FI^2$ \\
	
	$\I \cdot \F \cdot \ThII^2 \cdot \FI$ &%
	
	$\I \cdot \TII \cdot \F \cdot \ThII \cdot \FT$ &%
	
	$\I \cdot \III \cdot \ThII \cdot \FI \cdot \FT$ \\
	
	$\I \cdot \III \cdot \TII \cdot \FT^2$ &%
	
	$\TII \cdot \ThII \cdot \FI \cdot \FT$ &%
	
	$\ThII^2 \cdot \FI^2$ \\
	
	$\TII \cdot \FI^2 \cdot \ThTI$ &
	
	$\TII^2 \cdot \FT^2$ &%

	$\TII^2 \cdot \FI \cdot \FTI$  \\
	
	$\I^4 \cdot \III^2 \cdot \F \cdot \FII$ &%
	
	$\I^4 \cdot \III \cdot \F^2 \cdot \TTI$ &%
	
	$\I^3 \cdot \II \cdot \III^2 \cdot \F \cdot \FI$ \\
	
	$\I^3 \cdot \II \cdot \III \cdot \TII \cdot \F^2$ &%
	
	$\I \cdot \III \cdot \FI^2 \cdot \ThTI$ &%
	
	$\I^2 \cdot \II \cdot \III \cdot \TII \cdot \F \cdot \FI$ \\
	
	$\I^2 \cdot \II \cdot \TII^2 \cdot \F^2$ &%
	
	$\I^2 \cdot \III^2 \cdot \TI \cdot \F \cdot \FI$ &%
	
	$\I^2 \cdot \III \cdot \TI \cdot \TII \cdot \F^2$ \\
	
	$\I^2 \cdot \II \cdot \III^2 \cdot \FI^2$ &%
	
	$\I^3 \cdot \III^2 \cdot \FI \cdot \FII$ &%
	
	$\I^3 \cdot \III \cdot \F^2 \cdot \ThTI$ \\
	
	$\I^3 \cdot \III \cdot \TII \cdot \F \cdot \FII$ &%
	
	$\I^3 \cdot \III \cdot \F \cdot \TTI \cdot \FI$ &%
	
	$\I^3 \cdot \TII \cdot \F^2 \cdot \TTI$ \\
	
	$\I^3 \cdot \III^2 \cdot \F \cdot \FTI$ &%
	
	$\I \cdot \II \cdot \TII^2 \cdot \F \cdot \FI$ &%
	
	$\I \cdot \III^2 \cdot \TI \cdot \FI^2$ \\
	
	$\I \cdot \III \cdot \TI \cdot \TII \cdot \F \cdot \FI$ &%
	
	$\I \cdot \TI \cdot \TII^2 \cdot \F^2$ &%
	
	$\I \cdot \II \cdot \III \cdot \TII \cdot \FI^2$ \\
	
	$\I^2 \cdot \TII \cdot \F \cdot \TTI \cdot \FI$ &%
	
	$\I^2 \cdot \III \cdot \TII \cdot \FI \cdot \FII$ &%
	
	$\I^2 \cdot \III \cdot \TTI \cdot \FI^2$ \\
	
	$\I^2 \cdot \TII \cdot \F^2 \cdot \ThTI$ &%
	
	$\I^2 \cdot \TII^2 \cdot \F \cdot \FII$ &%
	
	$\I^2 \cdot \III \cdot \F \cdot \FI \cdot \ThTI$ \\
	
	$\I^2 \cdot \III \cdot \TII \cdot \F \cdot \FTI$ &%
	
	$\I^2 \cdot \III^2 \cdot \FI \cdot \FTI$ &%
	
	$\II \cdot \TII \cdot \TII \cdot \FI \cdot \FI$ \\
	
	$\III \cdot \TI \cdot \TII \cdot \FI^2$ &%
	
	$\TI \cdot \TII^2 \cdot \F \cdot \FI$ &%
	
	$\I \cdot \TII \cdot \TTI \cdot \FI^2$ \\
	
	$\I \cdot \TII^2 \cdot \FI \cdot \FII$ &%
	
	$\I \cdot \TII \cdot \F \cdot \FI \cdot \ThTI$ &%
	
	$\I \cdot \TII^2 \cdot \F \cdot \FTI$ \\
	
	&%
	
	$\I \cdot \III \cdot \TII \cdot \FI \cdot \FTI$ &%

\end{tabular}
\caption{Problems of type one.\label{fig: type 1}}
\end{figure}

\section{Type two}

Here, we study $\I^4 \cdot \III^3 \cdot \FT=4$ and all conclusions about this problem will also apply to the problems in Figure \ref{fig: type 2}. By Theorem \ref{thm: relation 2}, the Galois group of each of these problems is a subgroup of $\g(\I^4) \wr S_2$. By Lemma \ref{lem: aux probs}, this has Galois group $S_2$, thus we see that the Galois group of each of these problems is a subgroup of $S_2 \wr S_2 \cong D_4$. To see that this is the Galois group, we use Algorithm \ref{alg:Frob} yielding results similar to Table \ref{table: D4}.

\begin{figure}[h]
	\centering
	\begin{tabular}{lll}
		$\I^3 \cdot \III^2 \cdot \TII \cdot \FTh$ &%
		
		$\I^3 \cdot \III^3 \cdot \FThI$ &%
		
		$\I^2 \cdot \III \cdot \TII^2 \cdot \FTh$ \\
		
		$\I^2 \cdot \III^2 \cdot \TII \cdot \FThI$ &%
		
		$\I \cdot \TII^3 \cdot \FTh$ &%
		
		$\I \cdot \III \cdot \TII^2 \cdot \FThI$ \\
		
		$\I^4 \cdot \III^3 \cdot \FT$ &
		
		$\TII^3 \cdot \FThI$ &%
	\end{tabular}
	\caption{Problems of type two.\label{fig: type 2}}
\end{figure}

\section{Type three}
Here, we study the Schubert problems $\FTT \cdot \TT \cdot \III^2 \cdot \I^2, \FTT \cdot \TT \cdot \TII \cdot \III \cdot \I$, and $\FTT \cdot \TT \cdot \TII^2$, all of which have four solutions. By Theorem \ref{thm: relation 5}, the Galois group of each of these problems is a subgroup of $S_2 \wr S_2 \cong D_4$. To see that this is the Galois group, we use Algorithm \ref{alg:Frob} yielding results similar to Table \ref{table: D4}.

\section{Type four}

Here, we study $\I \cdot \T \cdot \III \cdot \TT \cdot \F \cdot \ThTh=4$ and all conclusions about this problem will also apply to the problems in Figure \ref{fig: type 4}. By Theorem \ref{thm: relation 4}, the Galois group of each of these problems is a subgroup of $\g(\I^2\cdot\T^2) \wr S_2$.  In $\gr(2,5)$, $\I^2 \cdot \T^2$ reduces to the problem $\I^4$ in $\gr(2,4)$ by Lemma \ref{lem: deficient}. By Lemma \ref{lem: aux probs}, this has Galois group $S_2$, thus we see that the Galois group of each of these problems is a subgroup of $S_2 \wr S_2 \cong D_4$. To see that this is the Galois group, we use Algorithm \ref{alg:Frob} yielding results similar to Table \ref{table: D4}.

\begin{figure}
	\centering
	\begin{tabular}{lll}
		& $\I \cdot \T \cdot \III \cdot \TT \cdot \F \cdot \ThTh$ & \\
		$\T \cdot \TII \cdot \TT \cdot \F \cdot \ThTh$ &%
		
		$\T \cdot \III \cdot \TT \cdot \FI \cdot \ThTh$ &%
		
		$\I \cdot \TT \cdot \F \cdot \ThII \cdot \ThTh$ \\
		
		$\I \cdot \III \cdot \TT \cdot \ThTh \cdot \FT$ &
		
		$\TII \cdot \TT \cdot \ThTh \cdot \FT$&
		
		$\TT \cdot \ThII \cdot \FI \cdot \ThTh$
	\end{tabular}
	\caption{Problems of type four.\label{fig: type 4}}
\end{figure}

\section{Type five}
Here, we study $\ThThT \cdot \TT^2 \cdot \III \cdot \I$ and $\ThThT \cdot \TT^2 \cdot \TII$, all of which have four solutions. By Theorem \ref{thm: relation 6}, the Galois group of each of these problems is a subgroup of $S_2 \wr S_2 \cong D_4$. To see that this is the Galois group, we use Algorithm \ref{alg:Frob} yielding results similar to Table \ref{table: D4}.

\section{Type six}
Here, we study $\ThTh \cdot \ThT \cdot \F \cdot \III \cdot \I^2, \ThTh \cdot \ThT \cdot \F \cdot \TII \cdot \I, \ThTh \cdot \ThT \cdot \FI \cdot \III \cdot \I$, $\ThTh \cdot \ThT \cdot \FI \cdot \TI$, $\ThThI \cdot \F \cdot \TT \cdot \III \cdot \I^2, \ThThI \cdot \F \cdot \TT \cdot \TII \cdot \I, \ThThI \cdot \FI \cdot \TT \cdot \III \cdot \I$, and $\ThThI \cdot \FI \cdot \TT \cdot \TII$, all of which have four solutions. By Theorem \ref{thm: relation 7}, the Galois group of each of these problems is a subgroup of $S_2 \wr S_2 \cong D_4$. To see that this is the Galois group, we use Algorithm \ref{alg:Frob} yielding results similar to Table \ref{table: D4}.

\section{Type seven}

Here, we study $\I^4 \cdot \T \cdot \III^2 \cdot \F^2=6$ and all conclusions about this problem will also apply to the problems in Figure \ref{fig: type 8}. By Theorem \ref{thm: relation 1}, the Galois group of each of these problems is a subgroup of $\g(\I^4\cdot\T) \wr S_2$. By Lemma \ref{lem: aux probs}, this has Galois group $S_3$, thus we see that the Galois group of each of these problems is a subgroup of $S_3 \wr S_2$. To see that this is the Galois group, we use Algorithm \ref{alg:Frob} yielding results similar to Table \ref{table: S3S2}.

\begin{table}
	\centering
	\caption{Example of Algorithm \ref{alg:Frob} output when the Galois group is $S_3 \wr S_2$.}\label{table: S3S2}
\begin{tabular}{|c|c|c|c|} 
	\hline 
	\multicolumn{4}{|c|}{\rule{0pt}{14pt}Cycles in $\g( \I^4 \cdot \T \cdot \III^2 \cdot \F^2)$} \\
	\multicolumn{4}{|c|}{found in 49557 samples} \\
	\hline 
	Cycle & Observed & Empirical & Number  \\
	Type & Frequency & Fraction & in $S_3 \wr S_2$ \\
	\hline
	(6) &  8275 & 12.0225 & 12 \\
	\hline
	(4,2) &  12451 & 18.0897& 18 \\
	\hline
	(3,3) &  2703 & 3.9271 & 4 \\
	\hline
	(3,2,1) &  8194 & 11.9048 & 12 \\
	\hline
	(3,1,1,1) &  2667 & 3.8748 & 4 \\
	\hline
	(2,2,2) &  4105 & 5.9640 & 6 \\
	\hline
	(2,2,1,1) &  6299 & 9.1516 & 9 \\
	\hline
	(2,1,1,1,1) &  4172 & 6.0614 & 6 \\
	\hline
	(1,1,1,1,1,1) &  691 & 1.004 & 1 \\
	\hline
\end{tabular} 
\end{table}

\begin{figure}
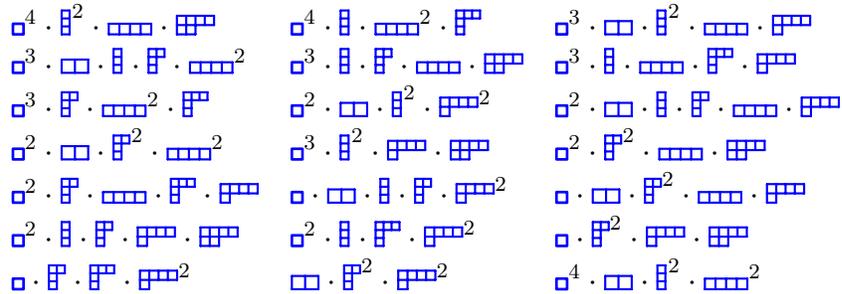

	\centering
\begin{tabular}{lll}
	$\I^4 \cdot \III^2 \cdot \F \cdot \FT $ &%
	
	$\I^4 \cdot \III \cdot \F^2 \cdot \ThII $ &%
	
	$\I^3 \cdot \T \cdot \III^2 \cdot \F \cdot \FI $ \\
	
	$\I^3 \cdot \T \cdot \III \cdot \TII \cdot \F^2 $ &%
	
	$\I^3 \cdot \III \cdot \TII \cdot \F \cdot \FT $ &%
	
	$\I^3 \cdot \III \cdot \F \cdot \ThII \cdot \FI $ \\
	
	$\I^3 \cdot \TII \cdot \F^2 \cdot \ThII $ &%
	
	$\I^2 \cdot \T \cdot \III^2 \cdot \FI^2 $ &%
	
	$\I^2 \cdot \T \cdot \III \cdot \TII \cdot \F \cdot \FI $ \\
	
	$\I^2 \cdot \T \cdot \TII^2 \cdot \F^2 $ &%
	
	$\I^3 \cdot \III^2 \cdot \FI \cdot \FT $ &%
	
	$\I^2 \cdot \TII^2 \cdot \F \cdot \FT $ \\
	
	$\I^2 \cdot \TII \cdot \F \cdot \ThII \cdot \FI $ &%
	
	$\I \cdot \T \cdot \III \cdot \TII \cdot \FI^2 $ &%
	
	$\I \cdot \T \cdot \TII^2 \cdot \F \cdot \FI $ \\
	
	$\I^2 \cdot \III \cdot \TII \cdot \FI \cdot \FT $ &%
	
	$\I^2 \cdot \III \cdot \ThII \cdot \FI^2 $ &%
	
	$\I \cdot \TII^2 \cdot \FI \cdot \FT $ \\
	
	$\I \cdot \TII \cdot \ThII \cdot \FI^2 $ &%
	
	$\T \cdot \TII^2 \cdot \FI^2 $ &%
	
	$\I^4 \cdot \T \cdot \III^2 \cdot \F^2$
\end{tabular}
\caption{Problems of type seven.\label{fig: type 8}}
\end{figure}

\section{Type eight}

Here, we study $\I^4 \cdot \III^3 \cdot \Th \cdot \F=6$ and all conclusions about this problem will also apply to the problems in Figure \ref{fig: type 9}. By Theorem \ref{thm: relation 3}, the Galois group of each of these problems is a subgroup of $\g(\I^4) \wr S_3$. By Lemma \ref{lem: aux probs}, this has Galois group $S_2$, thus we see that the Galois group of each of these problems is a subgroup of $S_2 \wr S_3$. To see that this is the Galois group, we use Algorithm \ref{alg:Frob} yielding results similar to Table \ref{table: S2S3}.

\begin{table}
	\centering
	\caption{Example of Algorithm \ref{alg:Frob} output when the Galois group is $S_2 \wr S_3$.}\label{table: S2S3}
\begin{tabular}{|c|c|c|c|} 
	\hline 
	\multicolumn{4}{|c|}{\rule{0pt}{14pt}Cycles in $\g(\I \cdot \III \cdot \TII^2 \cdot \ThI \cdot \F)$} \\
	\multicolumn{4}{|c|}{found in 49501 samples} \\
	\hline 
	Cycle & Observed & Empirical & Number  \\
	Type & Frequency & Fraction & in $S_2 \wr S_3$ \\
	\hline
	(6) &  8256 & 8.0057 & 8 \\
	\hline
	(4,2) &  6168 & 5.9810 & 6 \\
	\hline
	(4,1,1) &  6204 & 6.0159 & 6 \\
	\hline
	(3,3) &  8082 & 7.8369 & 8 \\
	\hline
	(2,2,2) &  7264 & 7.0437 & 7 \\
	\hline
	(2,2,1,1) &  9407 & 9.1217 & 9 \\
	\hline
	(2,1,1,1,1) &  3071 & 2.9779 & 3 \\
	\hline
	(1,1,1,1,1,1) &  1049 & 1.0172 & 1\\
	\hline
\end{tabular}
\end{table}

\begin{figure}
	\centering
	\begin{tabular}{lll}
		$\I^3 \cdot \III^3 \cdot \Th \cdot \FI $&%
		
		$\I^3 \cdot \III^2 \cdot \Th \cdot \TII \cdot \F $ &%
		
		$\I^3 \cdot \III^3 \cdot \ThI \cdot \F $ \\
		
		$\I^2 \cdot \III^2 \cdot \TII \cdot \ThI \cdot \F $ &%
		
		$\I^2 \cdot \III^2 \cdot \Th \cdot \TII \cdot \FI $ &%
		
		$\I^2 \cdot \III^3 \cdot \ThI \cdot \FI $ \\
		
		$\I^2 \cdot \III \cdot \Th \cdot \TII^2 \cdot \F $ &%
		
		$\I \cdot \III \cdot \TII^2 \cdot \ThI \cdot \F $ &%
		
		$\I \cdot \III^2 \cdot \TII \cdot \ThI \cdot \FI $ \\
		
		$\I \cdot \Th \cdot \TII^3 \cdot \F $ &%
		
		$\I \cdot \III \cdot \Th \cdot \TII^2 \cdot \FI $ &%
		
		$\TII^3 \cdot \ThI \cdot \F $ \\
		
		$\III \cdot \TII^2 \cdot \ThI \cdot \FI $ &%
		
		$\Th \cdot \TII^3 \cdot \FI $ &%
		
		$\I^4 \cdot \III^3 \cdot \Th \cdot\F = 6$
	\end{tabular}
	\caption{Problems of type eight.\label{fig: type 9}}
\end{figure}

\section{Type nine}
Here, we study $\I^2 \cdot \III^2 \cdot \TT^2 \cdot \F=6$ and all conclusions about this problem will also apply to the problems in Figure \ref{fig: type 10}. By Theorem \ref{thm: relation 9}, the Galois group of each of these problems is a subgroup of $\g(\I^4) \wr S_3$. By Lemma \ref{lem: aux probs}, this has Galois group $S_2$, thus we see that the Galois group of each of these problems is a subgroup of $S_2 \wr S_3$. To see that this is the Galois group, we use Algorithm \ref{alg:Frob} yielding results similar to Table \ref{table: S2S3}.

\begin{figure}
	\centering
	\begin{tabular}{lll}
		$\I^2 \cdot \III^2 \cdot \TT^2 \cdot \F$ &
		
		$\I \cdot \III \cdot \TII \cdot \TT^2 \cdot \F $ &%
			
		$\I \cdot \III^2 \cdot \TT^2 \cdot \FI $ \\
			
		$\TII^2 \cdot \TT^2 \cdot \F $ &
			
		$\III \cdot \TII \cdot \TT^2 \cdot \FI $ 
	\end{tabular}
	\caption{Problems of type nine.\label{fig: type 10}}
\end{figure}

\section{Type ten}

Here, we study $\I^3 \cdot \III \cdot \TT \cdot \F \cdot \ThTh=6$ and all conclusions about this problem will also apply to the problems in Figure \ref{fig: type 11}. By Theorem \ref{thm: relation 4}, the Galois group of each of these problems is a subgroup of $\g(\I^4\cdot\T) \wr S_2$. By Lemma \ref{lem: aux probs}, this has Galois group $S_3$, thus we see that the Galois group of each of these problems is a subgroup of $S_3 \wr S_2$. To see that this is the Galois group, we use Algorithm \ref{alg:Frob} yielding results similar to Table \ref{table: S3S2}.

\begin{figure}
	\centering
	\begin{tabular}{lll}
		& $\I^3 \cdot \III \cdot \TT \cdot \F \cdot \ThTh$ & \\
		
		$\I^2 \cdot \TII \cdot \TT \cdot \F \cdot \ThTh $ &%
		
		$\I^2 \cdot \III \cdot \TT \cdot \FI \cdot \ThTh $ &%
		
		$\I \cdot \TII \cdot \TT \cdot \FI \cdot \ThTh $ %
	\end{tabular}
	\caption{Problems of type ten.\label{fig: type 11}}
\end{figure}

\section{Type eleven}

Here, we study $\I^6 \cdot \III^2 \cdot \F^2 =10$ and all conclusions about this problem will also apply to the problems in Figure \ref{fig: type 12}. By Theorem \ref{thm: relation 1}, the Galois group of each of these problems is a subgroup of $\g(\I^6) \wr S_2$. By Lemma \ref{lem: aux probs}, this has Galois group $S_5$, thus we see that the Galois group of each of these problems is a subgroup of $S_5 \wr S_2$. To see that this is the Galois group, we use Algorithm \ref{alg:Frob} yielding results similar to Table \ref{table: S5S2}.

\begin{table}
	\centering
	\caption{Example of Algorithm \ref{alg:Frob} output when the Galois group is $S_5 \wr S_2$.}\label{table: S5S2}
	\begin{tabular}{|c|c|c|c|} 
		\hline 
		\multicolumn{4}{|c|}{\rule{0pt}{12pt}Cycles in $\g(\I^3 \cdot \III \cdot \TII \cdot \FI^2)$} \\
		\multicolumn{4}{|c|}{found in 49510 samples} \\
		\hline 
		Cycle & Observed & Empirical & Number  \\
		Type & Frequency & Fraction & in $S_5 \wr S_2$ \\
		\hline
		(10) &  4976 & 2894.5425 & 2880 \\
		\hline
		(8,2) &  6209 & 3611.7794 & 3600 \\
		\hline
		(6,4) &  4226 & 2458.2669 & 2400 \\
		\hline
		(6,2,2) &  4061 & 2362.2854 & 2400 \\
		\hline
		(5,5) &  1018 & 592.1713 & 576 \\
		\hline
		(5,4,1) &  2383 & 1386.1927 & 1440 \\
		\hline
		(5,3,2) &  1636 & 951.6623 & 960 \\
		\hline
		(5,3,1,1) &  1641 & 954.5708 & 960 \\
		\hline
		(5,2,2,1) &  1321 & 768.4266 & 720 \\
		\hline
		(5,2,1,1,1) &  798 & 464.1971 & 480 \\
		\hline
		(5,1,1,1,1,1) &  85 & 49.4446 & 48 \\
		\hline
		(4,4,2) &  3150 & 1832.3571 & 1800 \\
		\hline
		(4,4,1,1) &  1542 & 896.9824 & 900 \\
		\hline
		(4,3,2,1) &  2007 & 1167.4732 & 1200 \\
		\hline
		(4,3,1,1,1) &  2040 & 1186.6692 & 1200 \\
		\hline
		(4,2,2,2) &  2068 & 1202.9570 & 1200 \\
		\hline
		(4,2,2,1,1) &  1521 & 884.7667 & 900 \\
		\hline
		(4,2,1,1,1,1) &  1053 & 612.5308 & 600 \\
		\hline
		(4,1,1,1,1,1,1) &  103 & 59.9151 & 60 \\
		\hline
		(3,3,2,2) &  678 & 394.3930 & 400 \\
		\hline
		(3,3,2,1,1) &  1388 & 807.4005 & 800 \\
		\hline
		(3,3,1,1,1,1) &  680 & 395.5564 & 400 \\
		\hline
		(3,2,2,2,1) &  1017 & 591.5896 & 600 \\
		\hline
		(3,2,2,1,1,1) &  1735 & 1009.2505 & 1000 \\
		\hline
		(3,2,1,1,1,1,1) &  752 & 437.4389 & 440 \\
		\hline
		(3,1,1,1,1,1,1,1) &  74 & 43.0456 & 40 \\
		\hline
		(2,2,2,2,2) &  206 & 119.8303 & 120 \\
		\hline
		(2,2,2,2,1,1) &  371 & 215.8108 & 225 \\
		\hline
		(2,2,2,1,1,1,1) &  550 & 319.9353 & 300 \\
		\hline
		(2,2,1,1,1,1,1,1) &  192 & 111.6865 & 130 \\
		\hline
		(2,1,1,1,1,1,1,1,1) &  28 & 16.2876 & 20 \\
		\hline
		(1,1,1,1,1,1,1,1,1,1) &  1 & 0.6330 & 1 \\
		\hline
	\end{tabular} 
\end{table}

\begin{figure}
	\centering
	\begin{tabular}{lll}
		$\I^2 \cdot \TII^2 \FI^2$ &%
		
		$\I^3 \cdot \III \cdot \TII \cdot \FI^2$ &%
		
		$\I^3 \cdot \TII^2 \cdot \F \cdot \FI$ \\
		
		$\I^4 \cdot \III^2 \cdot \FI^2$ &%
		
		$\I^4 \cdot \III \cdot \TII \cdot \F \cdot \FI$ &%
		
		$\I^4 \cdot \TII^2 \cdot \F^2$ \\
		
		$\I^5 \cdot \III \cdot \TII \cdot \F^2$&
		
		$\I^5 \cdot \III^2 \cdot \F \cdot
		\FI$ &%
		
		$\I^6 \cdot \III^2 \cdot \F^2$
	\end{tabular}
	\caption{Problems of type eleven.\label{fig: type 12}}
\end{figure}

\chapter{CONCLUSION \label{cha:Summary}}

In general, we expect the Galois group of a Schubert problem to be the full symmetric group. When it is not, the group reflects some intrinsic geometric structure in the problem's set of solutions. We explored all of the reduced Schubert problems in $\gr(4,9)$ and found that the Galois groups of all except $148$ of them are at least alternating. Of the problems whose Galois groups are at least alternating, we were able to determine that over $80\%$ of them actually have the full symmetric group as their Galois group. For the $148$ deficient problems, we were able to completely determine their Galois groups. We also described three different structures, each of which may be used to build infinite families of examples of Schubert problems with deficient Galois group.

We believe that, with some more work, the other structures used to classify the deficient problems in $\gr(4,9)$ will be generalized to other Grassmannians as well. Moreover, the techniques developed in studying the Schubert problems in $\gr(4,9)$ are readily adapted for studying Schubert problems in other spaces. We are currently developing the computational structure that will be needed to do a thorough search of the Schubert problems in $\gr(4,10)$. We are also working to expand the scope of our current software so that it will be capable of exploring Schubert problems in more general flag varieties. This is part of a long-term project devoted to understanding the Galois groups of Schubert problems in general.

\let\oldbibitem\bibitem
\renewcommand{\bibitem}{\setlength{\itemsep}{0pt}\oldbibitem}
\bibliographystyle{amsplain}

\phantomsection
\addcontentsline{toc}{chapter}{REFERENCES}

\renewcommand{\bibname}{{\normalsize\rm REFERENCES}}

\bibliography{myReference}

\end{document}